\documentclass[11pt]{article}
\usepackage{amsfonts}
\usepackage{latexsym,amssymb}
\usepackage{amsmath, amsbsy}
\usepackage{amsopn, amstext}
\usepackage{graphicx, color, epstopdf}
\usepackage{threeparttable}
\usepackage{float}
\usepackage{multirow}

\newtheorem{lemma}{Lemma}
\newtheorem{theorem}{Theorem}
\newtheorem{remark}{Remark}

\newtheorem{proposition}{Proposition}

\numberwithin{equation}{section}
 \numberwithin{Lem}{section}
 \numberwithin{Defi}{section}
 \numberwithin{Theo}{section}
 \numberwithin{Rem}{section}
  \numberwithin{Coro}{section}
  \numberwithin{Fig}{section}

\def\NN{\hbox{\rlap{I}\kern.16em N}}
\def\NC{\hbox{\rlap{\kern.24em\raise.1ex\hbox
                  {\vrule height1.3ex width.9pt}}C}}

\def \bZ {\mathbb Z}

\title{ A $C^1$ Petrov-Galerkin method and Gauss collocation method for 1D general elliptic problems and superconvergence}

\author{Waixiang Cao\footnote{School of Mathematical Science, Beijing Normal University, Beijing
l00875,  China. ({\tt caowx@bnu.edu.cn}).   Research was supported in part by NSFC grant No.11871106.}
\and Lueling Jia
\footnote{Beijing Computational Science Research Center,
Beijing, 100193, China. ({\tt lljia@csrc.ac.cn}).}
\and Zhimin Zhang \footnote{Beijing Computational Science Research Center,
Beijing, 100193, China, Department of Mathematics, Wayne State University,
Detroit, MI 48202, USA.
({\tt zmzhang@csrc.ac.cn}). Research was supported in part by NSFC grants  No.11871092, U1930402.}    }

\date{}

\begin{document}

\maketitle

\begin{abstract}
   In this paper, we present and study  $C^1$ Petrov-Galerkin and Gauss collocation methods with arbitrary polynomial degree $k$ ($\ge 3$)  for one-dimensional elliptic equations. We prove that,  the solution and its derivative approximations converge with rate $2k-2$ at all grid points;
    and the solution approximation is superconvergent at all interior roots of a special Jacobi polynomial   of degree $k+1$ in each element, the first-order derivative approximation is superconvergent at all interior $k-2$ Lobatto points, and the second-order derivative approximation is superconvergent at $k-1$ Gauss points, with an order of $k+2$, $k+1$, and $k$, respectively.
  As a by-product,  we prove that both the Petrov-Galerkin solution and the Gauss collocation solution
  are superconvergent towards a particular Jacobi projection of the exact solution in $H^2$, $H^1$, and $L^2$ norms.
  All theoretical findings are confirmed by numerical experiments.
\end{abstract}

 
  \vskip 5pt \noindent {\bf AMS:} {65N30,65N35,65N12,65N15}

 \section{Introduction}
   Superconvergence   phenomenon means that the convergent rate exceeds the best
possible global rate at some special points. Those points are called superconvergent points.
During the past several decades,  the subject has attracted much attention from the scientific and engineering computing community, and it is well understood for the
 $C^0$ finite element method (see, e.g.,
  \cite{Babuska1996,Bramble.Schatz.math.com,Chen.C.M2001,Chen.C.M_Huang.Y.Q2001,ewing-lazarov-wang,Neittaanmaki1987,krizek-ns,Lin-Yan,Schatz.Wahlbin1996,V.Thomee.math.comp,Wahlbin,Zhu.QD;LinQ1989}), the $C^0$ finite volume method (see, e.g., \cite{Cai.Z1991,Cao;Zhang;Zou2012,Cao;zhang;zou:2kFVM,Chou_Ye2007,Xu.J.Zou.Q2009}), the discontinuous Galerkin method (see, e.g.,
\cite{Adjerid;Massey2006,Adjerid;Weinhart2009,Adjerid;Weinhart2011,Chen;Shu:JCP2008,Chen;Shu:SIAM2010,Guo_zhong_Qiu2013,Xie;Zhang2012,Yang;Shu:SIAM2012,Zhang;xie;zhang2009}),
and the spectral Galerkin method (see, e.g., \cite{zhang,zhang2008}).
Here by $C^0$ element methods we mean that the approximation space is continuous while its derivative function space is not continuous.  As comparison, the relevant study for $C^1$ element methods (i.e., both the approximation space and its derivative function space are continuous) is lacking. Only very special and simple cases have been
discussed (see. e.g., \cite{Wahlbin,Bialeck,Bhal}).

 Comparing with  continuous Galerkin (or $C^0$ element) and discontinuous Galerkin (DG) methods,  the most attractive feature of $C^1$ element methods is the continuity of the derivative approximation across the element interface.
As early as 1995, Wahlbin investigated the superconvergence of $C^1$  Galerkin (not Petrov Galerkin) and spline Galerkin methods in \cite{Wahlbin} for two-point boundary value problems and  established a mathematical theory to find superconvergence points for the $C^1$ finite element solution under the locally uniform mesh assumption.
It was proved in \cite{Wahlbin} that the function value approximation of the $k$-th $C^1$ Galerkin method
is superconvergent  with order $k+2$ at zeros of a special polynomial,
and the derivative error is $k+1$-th order superconvergent at grid points as well as element  mid-point when $k$ is odd.
While for even $k$, the superconvergence behavior changes: the function value approximation is superconvergent at interior Lobatto points, mesh points, and element mid-points, and  the derivative is superconvergent at  the Gauss points.
All those superconvergence rates are one order higher than the counterpart optimal
convergence  rates and the superconvergence results are valid in case that the mesh is locally uniform.
However, the generalization of the superconvergence analysis to quasi-uniform meshes is not straightforward.
In 1999, Bialeck \cite{Bialeck} studied piecewise Hermite bi-cubic orthogonal spline collocation solution of  the Poisson equation on rectangular mesh and proved a fourth-order accuracy of  the first order partial derivatives of the collocation solution at the partition nodes.
Only recently, Bhal and Danumjaya in \cite{Bhal}  presented  a cubic spline collocation method for the one dimensional Helmholtz equation with discontinuous coefficients,  and proved a fourth-order accuracy
 for the function value approximation and for the first-order derivative value approximation at the grid points.

 In this paper, we present and study a $C^1$ Petrov-Galerkin method and Gauss collocation method for  elliptic equations in 1D.  The trail space
 %of the $C^1$ Petrov-Galerkin method
 is taken as the $C^1$ polynomial  space of degree not more than $k$,
 while the test space  of the $C^1$ Petrov-Galerkin method  is chosen  as  the $L^2$ polynomial  space of degree not more than $k-2$. As the reader may recall,
 the total  degrees of  freedom for the $C^1$ Petrov-Galerkin method is the same as that for the counterpart $C^0$ element method.
 The main purpose of our current work is to provide  a unified mathematical approach to establish the superconvergence theory of $C^1$ element methods.
 We prove that, for  general 1D elliptic equations,  the solution of the $C^1$ Petrov-Galerkin method is  superclose to
 a particular Jacobi projection of the exact solution and thus establish the following supreconvergece results at some  special points:
 1) both the function value and the first-order derivative approximations are superconvergent with order $2k-2$ at mesh nodes;
2) the function value approximation is superconvergent with order $k+2$ at roots of a generalized Jacobi polynomial;
3) the first-order derivative approximation is superconvergent with order $k+1$ at interior Lobatto points;
4) the second-order  derivative approximation is superconvergent with order $k$ at interior Gauss points.
  By interpreting the Gauss collocation method  as a Petrov-Galerkin method up to some higher-order numerical integration errors, we also prove that the Gauss-collocation solution inherits almost all the
   superconvergence properties from  the counterpart Petrov-Galerkin solution.

  The main contribution of this paper lies in that: in one hand,  we provide  a unified approach to establish the superconvergence theory of $C^1$ element methods and discover some new superconvergence phenomena,
  especially the $(2k-2)$-th convergence rate of the  derivative approximation at grid points and the superconvergence for the second order derivative approximation,  which is greatly different from the $C^0$ element method and DG method, even
  the $C^1$ finite element method in \cite{Wahlbin}; on the other hand,  all our superconvergence  results are
valid for non-uniform meshes. In other words, we improve the mesh condition from  locally uniform meshes in \cite{Wahlbin} to quasi-uniform meshes.  Furthermore, the superconvergence results for the $C^1$ Gauss collocation method can be viewed as the
generalization of the one presented in \cite{Bhal}. Actually,  the cubic spline collocation method  in
\cite{Bhal}  is a special case of our current $C^1$ Gauss collocation method in case of $k=3$.

 The rest of the paper is organized as follows. In section 2, we present a $C^1$ Petrov-Galerkin method and
 Gauss collocation method for elliptic equations under the one-dimensional setting. In section 3, we investigate  approximation
 properties and superconvergence properties of a special Jacobi projection of the exact solution, which is the basis to establish
 the superconvergence theory for $C^1$ element methods. In section 4 and section 5, we separately study the superconvergence behavior of
 $C^1$ Petrov-Galerkin and Gauss collocation methods, where superconvergence at the grid points (function and first order derivative value approximations), at interior roots of Jacobi polynomials (function  value approximation), at interior
 Lobatto points (first order derivative value approximation) and Gauss points (the second order derivative value approximation) are investigated.
 Numerical experiments supporting our theory are presented in section 6. Some
concluding remarks are provided in section 7.

   Throughout this paper,  we adopt standard notations for Sobolev spaces such as $W^{m,p}(D)$ on sub-domain $D\subset\Omega$ equipped with
    the norm $\|\cdot\|_{m,p,D}$ and semi-norm $|\cdot|_{m,p,D}$. When $D=\Omega$, we omit the index $D$; and if $p=2$, we set
   $W^{m,p}(D)=H^m(D)$,
   $\|\cdot\|_{m,p,D}=\|\cdot\|_{m,D}$, and $|\cdot|_{m,p,D}=|\cdot|_{m,D}$. Notation $A\lesssim B$ implies that $A$ can be
  bounded by $B$ multiplied by a constant independent of the mesh size $h$.
  $A\sim B$ stands for $A\lesssim B$ and $B\lesssim A$.

\section{$C^1$ Petrov-Galerkin methods and Gauss collocation methods}

   We consider the  following two-point boundary value problem
  \begin{eqnarray}\label{con_laws}
\begin{aligned}
   &-(\alpha u')' +\beta u'+\gamma u=f ,\ \ &&x\in \Omega= (a,b), \\
   &u(a)=u(b)=0, && \\
\end{aligned}
\end{eqnarray}
  where $\alpha>\alpha_0>0,\gamma-\frac{\beta'}{2} \ge 0, \gamma\ge 0$, $\alpha,\beta,\gamma\in L^{\infty}(\bar\Omega)$, and $f$ is real-valued function defined on $\bar\Omega$.
  For simplicity, we assume that $\alpha,\beta,\gamma$ are all constants.
  Other than technical complexity, there is no essential difficulty in analysis for variable coefficients as long as the above conditions are satisfied.

  Let  $a=x_{0}<x_{1}<\ldots<x_{N}$ be $N+1$ distinct points on the
  interval $\bar{\Omega}$.
   For all positive integers $r$, we define
  $\bZ_{r}=\{1,\ldots,r\}$ and denote by
\[
    \tau_j=(x_{j-1},x_{j}),\ j\in\bZ_N.
\]
  Let $h_j=x_{j}-x_{j-1}$,
  and $h =  \displaystyle\max_j\; h_j$. We assume  that the mesh
  is quasi-uniform, i.e., there exists a constant $c$ such that
\[
    h\le c h_j, j\in\bZ_N.
\]

  Define
\[
    V_h:=\{ v\in C^{1}(\Omega): \; v|_{\tau_j}\in P_k(\tau_j),\; j\in\bZ_N\}
^{}\]
  to be the $C^1$ finite element space, where $P_k, k\ge 3$ denotes the space of
  polynomials of degree not more than $k$. Let
\[
    V_h^0:=\{ v\in V_h: \; v(a)=v(b)=0\}.
^{}\]

   We adopt two numerical methods to solve the problem \eqref{con_laws}, i.e., the Petrov-Galerkin method and the Gauss collocation method.
  To establish the Petrov-Galerkin method, we
choose $V_h^0$ as our trail space and the piecewise polynomial space of degree $k-2$ as the
test space, which is defined as follows:
\[
    W_h:=\{ w\in L^2(\Omega): \; w|_{\tau_j}\in P_{k-2}(\tau_j),\; j\in\bZ_N\}.
\]

 % {\bf FEM}:  The finite element method is to find a $u_h\in V_h$ such that
% \begin{equation}\label{FEM}
%     a(u_h,v_h) :=  (\alpha u'_h,v_h')+(\beta u'_h+\gamma u_h, v_h)=(f,v_h),\ \ \forall v_h\in V_h.
% \end{equation}

 {\bf Petrov-Galerkin  method}: The Petrov-Galerkin method for solving \eqref{con_laws} is to find a $u_h\in V^0_h$ such that
 \begin{equation}\label{PG}
       (-\alpha u''_h,v_h)+(\beta u'_h+\gamma u_h, v_h)=(f,v_h),\ \ \forall v_h \in W_h.
 \end{equation}

 {\bf Gauss collocation method}: Given any $i\in \bZ_{N}$, we denote by
 $g_{im}, m\in \bZ_{k-1}$ the $k-1$ Gauss points in the interval $\tau_i$. That is,
 $\{g_{im}\}_{m=1}^{k-1}$ are zeros of the Legendre polynomial of degree $k-1$. Then
 the Gauss collocation method to \eqref{con_laws} is:  Find a $ \bar u_h\in V^0_h$ such that
 \begin{equation}\label{collocation}
     (-\alpha \bar u_h''+\beta \bar u'_h+\gamma  \bar u_h)(g_{im})=f(g_{im}), \ \ (i,m)\in\bZ_{N}\times \bZ_{k-1}.
 \end{equation}

  \section{Approximation and superconvergence  properties of the truncated Jacobi projection}

     In this section, we define a $C^1$ Jacobi projection of the exact solution and study the approximation and
     superconvergence properties of the Jacobi projection, which is of great importance to establish
     superconvergence results for the $C^1$ numerical solution, especially the discovery of superconvergence points.

 % \section{Truncated Jacobi projection of the exact solution}
     We begin with some preliminaries.  We first introduce the Jacobi polynomials.  The Jacobi polynomials, denote by $J_n^{r,l},\  r,l>-1$,
     are orthogonal with respect to the Jacobi weight function
     $\omega_{r,l}(s):=(1-s)^{r}(1+s)^{l}$ over $I:=(-1,1)$. That is,
 \[
     \int_{-1}^1J_n^{r,l}(s)J_m^{r,l}(s)\omega_{r,l}(s)ds
     =\kappa_n^{r,l}\delta_{mn},
 \]
 where $\delta$ denotes the Kronecker symbol and
 $$
  \kappa_n^{r,l}=\|J_n^{r,l}\|^2_{\omega_{r,l}}:=
  \frac{2^{r+l+1}\Gamma(n+r+1)\Gamma(n+l+1)}{(2n+r+l+1)\Gamma(n+1)\Gamma(n+r+l+1)}.
  %\int_{-1}^1J_n^{\alpha,\beta}(s)J_n^{\alpha,\beta}(s)\omega_{\alpha,\beta}(s)ds
 $$
   Here $\Gamma(n)$ denotes the Gamma function.
Note that when $r=l=0$, the Jacobi polynomial $J_n^{r,l}$ is reduced to the standard Legendre polynomial. That is
 $J_{n}^{0,0}(s)=L_{n}(s)$ with
  $L_{n}(s)$  being the  Legendre polynomial of degree $n$ over $[-1,1]$.
 We extend the definition of the classical Jacobi polynomials to the cases where
     both parameters $r,l \le -1$
 \begin{eqnarray}\label{Jacobi:1}
     J_n^{r,l}(s):=(1-s)^{-r}(1+s)^{-l}J_{n+r+l}^{-r,-l}(s),\ \ r,l\le -1.
 \end{eqnarray}
   It was proved in \cite{shen2011} (see Lemma 6.2) that the Jacobi
polynomials satisfy  the following derivative recurrence relation
\begin{equation}\label{jacobi:deri}
   \partial_sJ_n^{r,l}(s)
   =C_n^{r,l}J_{n-1}^{r+1,l+1}(s).
\end{equation}
  where
\begin{equation}
 C_n^{r,l}=\left\{\begin{array}{lll}
 -2(n+r+l+1), & \text{if} &r,l\le -1\\
-n, & \text{if} & r\le -1,l>-1,\ {\text or}\ r> -1,l\le -1,\\
\frac12(n+r+l+1),& \text{if} & r,l>-1.
\end{array}
\right.
\end{equation}
By taking  $r=l=-2$ in \eqref{Jacobi:1} and using the derivative recurrence relation \eqref{jacobi:deri},  we  obtain
 \begin{eqnarray*}
     J_n^{-2,-2}(s)&=&(1-s)^{2}(1+s)^{2}J_{n-4}^{2,2}(s)=\frac 2n(1-s)^{2}(1+s)^{2} \partial_s J_{n-3}^{1,1}(s)\\
           &=& \frac {4}{n(n-1)}(1-s)^{2}(1+s)^{2} \partial^2_{s} J_{n-2}^{0,0}(s)=\frac {4}{n(n-1)}(1-s)^{2}(1+s)^{2} \partial^2_{s} L_{n-2}(s).
 \end{eqnarray*}
    On the other hand,  we have, from \eqref{jacobi:deri}
 \begin{eqnarray}\label{orth:2}
   \partial_{s}^2J_n^{-2,-2}(s)=-2(n-3)\partial_sJ_{n-1}^{-1,-1}(s)=c_nJ_{n-2}^{0,0}(s)=c_nL_{n-2}(s),\ c_n=4(n-3)(n-2).
 \end{eqnarray}
     The above
   Jacobi polynomial plays an important role in our  later superconvergence analysis.

    Given any function $u\in C^1(\Omega)$, suppose $u(x)$ has the following Jacobi expansion in each element $\tau_i,i\in\bZ_N$
  \begin{eqnarray}\label{uu}
      u(x)|_{\tau_i}=H_3u(x)+\sum_{n=4}^{\infty} u_n \hat J_n^{-2,-2}(x),
  \end{eqnarray}
    where $\hat J_n^{-2,-2}(x)=J_n^{-2,-2}(\frac{2x-x_i-x_{i-1}}{h_i})=J_n^{-2,-2}(s),\ \ s\in [-1,1]$ is the Jacobi polynomial of degree $n$ over $\tau_i$,
    and $H_3u\in P_3$ denotes the Hermite  interpolation of $u$,  i.e.,
 \[
     \partial_x^m H_3u(x_i)=\partial_x^m u(x_i),\ \ \partial_x^m H_3u (x_{i-1})=\partial_x^mu(x_{i-1}), \ \ m=0,1.
 \]
   By \eqref{orth:2} and the orthogonality properties of the Legendre  polynomial, we have
 \begin{equation}\label{uun}
    u_n=\frac{h_i^2}{4c_n}\int_{\tau_i}(\partial_x^2uL_{i,n-2})(x)dx/\int_{\tau_i}L_{i,n-2}(x)L_{i,n-2}(x)dx.
 \end{equation}
   Here $c_n$ is the same as that in \eqref{orth:2} and $L_{i,n}(x)$ denotes the Legendre polynomial of degree $n$ over $\tau_i$, that is,
\[
    L_{i,n}(x)=L_{n}(\frac{2x-x_i-x_{i-1}}{h_i})=L_n(s),\ \ s\in [-1,1].
\]

    Now we define a  truncated Jacobi projection $u_I\in V_h$ of $u$ as follows:
  \begin{equation}\label{eq: interpolation}
      u_I(x)|_{\tau_i}:=\left\{\begin{array}{lll}
      H_3u(x)+\sum\limits_{n=4}^{k} u_n \hat J_n^{-2,-2}(x),& \text{if} &\ k\ge 4,\\
      H_3u(x),& \text{if} &\ k=3.
      \end{array}
  \right.
  \end{equation}
        We have the following orthogonal and approximation  properties for $u_I$.

 \begin{proposition}\label{proposition:1}
 Assume that $u\in W^{k+2,\infty}(\Omega)$ is the solution of \eqref{con_laws}, and $u_I$ is the Jacobi truncation projection of $u$ defined by \eqref{eq: interpolation}. Then the following orthogonality and approximation properties hold true.
  \begin{enumerate}
  \item Orthogonality:
\begin{equation}\label{ortho:3}
      \int_{\tau_i}(u-u_I)''v dx=0,\ \   \int_{\tau_i}(u-u_I)'v'dx=0, \ \  \int_{\tau_i}(u-u_I)v''=0,\ \ \forall v\in P_{k-2}(\tau_i).
 \end{equation}
   \item Optimal error estimates:
 \begin{equation}\label{optimaluI}
    \|u-u_I\|_{0,\infty,\tau_i}+h\|u-u_I\|_{1,\infty,\tau_i}\lesssim h^{k+1}|u|_{k+1,\infty,\tau_i}.
 \end{equation}
  \item Superconvergence of function value approximation on  roots of  $J_{k+1}^{-2,-2}$:
  \begin{equation}\label{approx:super12}
 \ \ (u-u_I)(x_i)=(u-u_I)(x_{i-1})=0,\ \  |(u-u_I) (l_{im}) | \lesssim   h^{k+2}|u|_{k+2,\infty,\tau_i},
\end{equation}
where $l_{im}$, $m=1,\cdots, k-3$ for $k\ge 4$ are interior roots of $J_{k+1}^{-2,-2}$ in $\tau_i$.
 \item
 Superconvergence of first order derivative value approximation on  Gauss-Lobatto points:
\begin{equation}\label{approx:super3}
 (u-u_I)'(x_i)=(u-u_I)'(x_{i-1})=0,\ \  |(u'-u_I)' (gl_{in})|\lesssim h^{k+1}|u|_{k+2,\infty,\tau_i},
\end{equation}
where  $gl_{in}$, $n=1,\cdots, k-2$ are interior roots of $\partial_x \hat J_{k+1}^{-2,-2}(x)=c_k \hat J_{k}^{-1,-1}(x)$ on $\tau_i$. That is,
$gl_{in}, i\le k-2$ are interior Gauss-Lobatto points of degree $k-2$.

\item Superconvergence of second order derivative value approximation on  Gauss   points:
\begin{equation}\label{approx:super4}
   |(u-u_I)'' (g_{in})| \lesssim h^{k}|u|_{k+2,\infty,\tau_i},
\end{equation}
where  $g_{in}$, $n\le k-1$ are interior roots of $L_{i,k-1}(x)$, i.e.,
 the  Gauss  points of degree $k-1$.

\end{enumerate}
 \end{proposition}
 \begin{proof}  First,  subtracting   \eqref{eq: interpolation} from \eqref{uu} yields that
\begin{equation}\label{uminusuI}
    \partial_x^m (u-u_I)(x)=\sum_{n=k+1}^{\infty} u_n  \partial_x^m \hat J_{n}^{-2,-2}(x),\ \ m=0,1,2.
\end{equation}
   Using  \eqref {orth:2} and  the orthogonal properties of  Legendre polynomials, we derive
\begin{equation} \label{ortho:1}
    \int_{\tau_i}(u-u_I)''v=0,\ \ \forall v\in P_{k-2}(\tau_i).
\end{equation}
  Noticing  that
\[
     \partial_x^m \hat J_{n}^{-2,-2}(x_i)= \partial_x^m \hat J_{n}^{-2,-2}(x_{i-1})=0,\ \ m=0,1,
\]
  we easily get
  \begin{eqnarray}
    \partial_x^m u_I (x_i)=\partial_x^mu(x_i),\ \ \partial_x^m  u_I (x_{i-1})=\partial_x^mu(x_{i-1}),\ \  m=0,1.
  \end{eqnarray}
   Consequently,   a simple  integration by parts and  \eqref{ortho:1} lead to
   \begin{eqnarray}
      \int_{\tau_i}(u-u_I)'v'=0,\ \  \int_{\tau_i}(u-u_I)v'' =0,\ \ \forall v\in P_{k-2}(\tau_i).
 \end{eqnarray}
  That is,
\[
    (u-u_I)\bot P_{k-4},\ \ (u-u_I)'\bot P_{k-3},\ \ (u-u_I)''\bot P_{k-2}.
\]
  Then \eqref{ortho:3} follows.

 %  Furthermore, we have the following approximation property for $u_I$.
% \[
%     \|u-u_I\|_{0,p}\lesssim h^{k+1}\|u\|_{k+1,p},\ \ p\ge 1.
% \]

  We now prove the approximation and superconvergence properties \eqref{optimaluI}-\eqref{approx:super4}.
   By a scaling from $\tau_i$ to $[-1,1]$ and a simple integration by parts for \eqref{uun},  we have
 \begin{eqnarray*}\label{coeff}
     u_n &= &\frac{(2n-3)}{2c_n}\int_{-1}^1 \partial_{s}^2 u(s) L_{n-2}(s) ds =\gamma_n\int_{-1}^1 \partial_{s}^2 u(s) \frac{d^{n-2}(1-s^2)^{n-2}}{ds^{n-2}} ds \\ \label{coeff:estimate}
           &= &(-1)^{n-2} \gamma_n\int_{-1}^1 \partial^{n}_{s} u(s) (1-s^2)^{n-2},\ \ \forall n\ge 4,
 \end{eqnarray*}
    where
 \[
    u(s)=u(\frac{2x-x_i-x_{i-1}}{h_i})=u(x),\ \ s\in[-1,1],x\in\tau_i,\quad \gamma_n=\frac{(-1)^n(2n-3)}{c_n2^{n-1}(n-2)!}.
 \]
  Noticing that
\[
    \partial^n_s u(s)=(\frac{h_i}{2})^n \partial^n_x u(x)=O(h^n),
\]
  we have
\begin{equation}\label{un}
    |u_n|\lesssim h^n|u|_{n,\infty},\ \ \forall n\ge 4.
\end{equation}
  Then \eqref{optimaluI} follows.
   At roots of  $J_{k+1}^{-2,-2}(x)$, there holds
\[
   |(u-u_I)(l_{im})|= \left| \sum_{n=k+2}^{\infty} u_n  \hat J_{n}^{-2,-2}(x)   \right|\lesssim h^{k+2}|u|_{k+2,\infty}.
\]
 This finishes the proof of \eqref{approx:super12}. Similarly, we can prove \eqref{approx:super3}-\eqref{approx:super4}.
 The proof is complete.
 \end{proof}

  \section{Superconvergence for $C^1$ Petrov-Galerkin methods}

  In this section, we  study superconvergence properties of the  $C^1$ Petrov-Galerkin method for \eqref{con_laws}.
  To this end, we begin with the introduction of  the bilinear form of the finite element method and some Green functions.

  First, we denote by $a(\cdot,\cdot)$ the bilinear form of the finite element method,  which is defined  as
\[
   a(u,v)= (\alpha u',v')-(\beta u,v')
   +(\gamma u, v),\ \ \forall u,v\in H^1(\Omega).
\]

 Second, given any $x\in \Omega$, let $G(x,\cdot)$ be the Green function for the problem \eqref{con_laws}.
% That is, $G(x,\cdot)$ satisfies
% \begin{eqnarray*}
%     && -(\alpha G')'-\beta G'+\gamma G=0,\ \ \forall y\in (a,x)\cup (x, b),\\
%     && G(x,x+0)=G(x,x-0),\ \ \alpha G'(x,x+0)-\alpha G'(x,x-0)=1.
% \end{eqnarray*}
  Then for any $v\in H^1(\Omega)$,
 \begin{equation}\label{eq:3}
    v(x)=a(v,G(x,\cdot)),\ \ \forall x\in \Omega.
 \end{equation}
    Especially, if $v(x)\in H_0^1(\Omega)$, then the Green function $G(x,\cdot)$
    satisfies $G(x,a)=G(x,b)=0$.
  Let $S_h$ be the $C^0$ finite element space, i.e.,
\[
    S_h=\{v\in C^0(\Omega): v|_{\tau_i}\in P_{k}, v(a)=v(b)=0, i\in\bZ_N\}.
 \]
   Denote by $G_h\in S_h$ the Galerkin approximation of $G(x,\cdot)$, that is,
 \begin{equation}\label{eq:4}
     v_h(x)=a(v_h,G_h)=a(v_h,G(x,\cdot)),\ \ \forall v_h\in S_h.
 \end{equation}

 % Let $G_1(x,\cdot)$ be the derivative typed Green function for the problem \eqref{con_laws} associated with the point $x$. That is,
% \begin{equation}\label{eq:d3}
%    v'(x)=a(v,G_1(x,\cdot)),\ \ \forall x\in \Omega,\ \ \forall v\in H^1.
% \end{equation}
%   Similarly, we can define the Galerkin approximation of $G_1(x,\cdot)$, which is denoted by $G_{1,h}$.

  Finally,  we use the following notations in the rest of this paper
 \[
     e_h:=u-u_h=\xi+\eta,\ \  \xi:=u_I-u_h,\ \ \eta:=u-u_I.
 \]

We have the following optimal error estimates for
    the $C^1$ Petrov-Galerkin method.

 \begin{lemma}\label{lemma:1}
 Assume that $u\in W^{k+1,\infty}(\Omega)$ is the solution of \eqref{con_laws}, and
 $u_h$ is the solution of \eqref{PG}. Then
 \begin{equation}\label{optimal:1}
     \|u-u_h\|_{0,\infty}\lesssim h^{k+1}|u|_{k+1,\infty},\ \ \|u-u_h\|_{1,\infty}\lesssim h^{k}|u|_{k+1,\infty},\ \
     \|u-u_h\|_{2}\lesssim h^{k-1}|u|_{k+1,\infty}.
 \end{equation}
 \end{lemma}
 \begin{proof}
   First, noticing that the exact solution $u$ also satisfy \eqref{PG}, we have
\begin{equation}\label{PG:error}
       (-\alpha e''_h,v_h)+(\beta e'_h+\gamma e_h, v_h)=0,\ \ \forall v_h\in W_h.
 \end{equation}
   Especially, we choose $v_h=-\xi'' $ in the above equation and using the orthogonal property of $\eta$ in \eqref{ortho:3} and \eqref{optimaluI} to get
 \begin{eqnarray}\label{eq:7}
     (\alpha \xi'', \xi'')+(\gamma\xi',\xi')-\frac{\beta}{2}(|\xi'(b)|^2-|\xi'(a)|^2)&=&(-\alpha \eta''+\beta\eta'+\gamma \eta,\xi'')\\\label{eq:5}
        &\lesssim& (|\beta| h^{k}+\gamma h^{k+1})\|\xi''\|_0|u|_{k+1,\infty}.
 \end{eqnarray}
   On the other hand, noticing that $e'_h(x)\in C^0(\Omega)\subset H^1(\Omega)$, we take $v=e_h'$ in \eqref{eq:3} and use the
   integration by parts  to obtain
 \begin{eqnarray*}
   e'_h(x_i)&=&a(e'_h,G(x_i,\cdot))=(\alpha e_h''-\beta e_h',G'(x_i,\cdot))
   +(\gamma e'_h,G(x_i,\cdot)) \\
  &=& (\alpha e_h''-\beta e_h'-\gamma e_h, G'(x_i,\cdot))\\
  &=&(\alpha e_h''-\beta e_h'-\gamma e_h, G'(x_i,\cdot)-{\cal I}_{k-2}G'(x_i,\cdot)).
 %  &\lesssim & h^{2(k-1)}\|u\|_{k+1,\infty}.
\end{eqnarray*}
   Here  ${\cal I}_{k-2} v$ denotes the $L^2$ projection of $v$ onto $P_{k-2}$.  Since the Green function
  $G(x_i,\cdot)\in C^{k}(\tau_j), j\in\bZ_N$ is bounded, we have
\begin{equation}\label{eh}
 |\xi'(x_i)| = |e'_h(x_i)|\lesssim h^{k-1}\|e_h\|_{2}\lesssim h^{2(k-1)}|u|_{k+1}+ h^{k-1}\|\xi\|_{2},\ \ i=0,\ldots, N.
\end{equation}
  Substituting \eqref{eh} into \eqref{eq:5} and using the Cauchy-Schwartz inequality yields
 \begin{eqnarray*}
     ( \alpha \xi'', \xi'')+(\gamma \xi',\xi')\le C(\beta h^{k}+\gamma h^{k+1}+h^{2(k-1)})^2|u|^2_{k+1,\infty}+(\frac 12+C_1h^{2(k-1)})\|\xi\|^2_{2},
 \end{eqnarray*}
   where $C, C_1$ are some positive constants independent of $h$.
   Consequently, when $h$ is sufficiently small, there holds
 \begin{equation}\label{superxi2}
     |\xi|_{2}+|\xi|_1\lesssim (|\beta| h^{k}+\gamma h^{k+1}+h^{2(k-1)})|u|_{k+1,\infty}.
 \end{equation}

   Similarly, we  choose $v_h=\xi\in H_0^1(\Omega)$ in \eqref{eq:4} and again use the integration by parts  to obtain
  \begin{eqnarray*}
     |\xi (x)|=|a( \xi,G_h)|
      &=&\left|( -\alpha  \xi''+\beta  \xi'+\gamma  \xi,G_h-\bar G_h)+( -\alpha  \xi''+\beta  \xi'+\gamma  \xi,\bar G_h)\right|\\
      &\le& h\|\xi\|_{2}\|G_h\|_{1}+|( -\alpha  \eta''+\beta  \eta'+\gamma  \eta,\bar G_h)|\\
      &\le & h\|\xi\|_{2}\|G_h\|_{1}+h^{k+1}|u|_{k+1,\infty}\|\bar G_h\|_{1,1}.
  \end{eqnarray*}
      Here  $\bar G_h|_{\tau_j}\in P_0(\tau_j)$ denotes the cell average of $G_h$.
 It has been proved in \cite{Chen.C.M2001} that
 \[
     \|G_h\|_{2,1}\lesssim 1,
 \]
   which yields, together with the embedding theory
 \[
    \|G_h\|_1\lesssim \|G_h\|_{2,1}\lesssim 1.
 \]
   Consequently,
 \[
    | \xi (x)|\lesssim h^{k+1}|u|_{k+1,\infty}+h\|\xi\|_{2}\lesssim h^{k+1}|u|_{k+1,\infty},
 \]
    and thus,
\[
    \| \xi\|_{0,\infty}\lesssim h^{k+1}|u|_{k+1,\infty},\ \  \| \xi\|_{1,\infty}\lesssim h^{-1}\|\xi\|_{0,\infty}\lesssim h^{k}|u|_{k+1,\infty}.
\]
  Then  \eqref{optimal:1} follows from the triangle inequality and the standard approximation theory.

 \end{proof}

  Now we are ready to present the superconvergence of the solution for the $C^1$ Petrov-Galerkin method.

 \begin{theorem}\label{theo:1}
 Assume that $u\in W^{k+2,\infty}(\Omega)$ is the solution of \eqref{con_laws}, and $u_h$ is the solution of \eqref{PG}.
The following superconvergence properties hold true.
  \begin{enumerate}

  \item  Supercloseness between the numerical solution and truncation projection in the $H^2$ norm:
\begin{equation}\label{pg:superclose}
    \|u_h-u_I\|_2\lesssim h^{k}|u|_{k+1,\infty}, \ \ {\rm if }\  \beta\neq 0, \  \|u_h-u_I\|_2\lesssim h^{k+1}|u|_{k+1,\infty},\ \ {\rm if \ } \beta= 0, \gamma\neq 0.
\end{equation}

  \item  if $\beta=\gamma =0$, then
\begin{equation}\label{superbeta0}
    u_h(x)=u_I(x), \ \ (u-u_h)(x_i)=0,\ \ (u-u_h)'(x_i)=0.
\end{equation}

\item Superconvergence for both function value and derivative value approximations at nodes:
\begin{equation}\label{super:nodes}
   |(u-u_h)(x_i)|\lesssim h^{2(k-1)}|u|_{k+1,\infty},\ \ |(u-u_h)'(x_i)|\lesssim h^{2(k-1)}|u|_{k+1,\infty},\ \ i\in\bZ_N.
\end{equation}
  \item Superconvergence of function value approximation on  interior  roots of  $J_{k+1}^{-2,-2}$ for $k\ge 4$:
  \begin{equation}\label{approx:super1:pg}
  |(u-u_h) (l_{im}) | \lesssim   h^{k+2}\|u\|_{k+2,\infty},
\end{equation}
where $l_{im}$, $m=1,\cdots, k-3$ are $k-3$ interior roots of $J_{k+1}^{-2,-2}$ in $\tau_i$.

 \item Superconvergence of first order derivative value approximation on  Gauss-Lobatto points:
\begin{equation}\label{super:derivlobatto}
   |(u-u_h)' (gl_{in})| \lesssim h^{k+1}\|u\|_{k+2,\infty},
\end{equation}
where
$gl_{in}, n\le k-2$ are interior Gauss-Lobatto points of degree $k-2$.

\item Superconvergence of second order derivative value approximation on  Gauss   points:
\begin{equation}\label{super:secondderivative}
   |(u-u_h)'' (g_{in})| \lesssim h^{k}\|u\|_{k+2,\infty},
\end{equation}
where  $g_{in}$, $n\le k-1$ are interior roots of $L_{i,k-1}(x)$, i.e.,
 the  Gauss  points of degree $k-1$.

\end{enumerate}
 \end {theorem}
 \begin{proof}
   First, \eqref{pg:superclose} follows directly from \eqref{superxi2}.
   Furthermore, there holds from \eqref{eq:5},
\[
    |u_h-u_I|_{2}+ |u_h-u_I|_{1}=0, \ \ {\rm if}\ \beta=\gamma=0,
\]
  which indicates that $u_h-u_I$ is a constant. Noticing that $(u_h-u_I)(a)=0$, we have
\[
    u_I=u_h.
\]
  Then \eqref{superbeta0} follows.

   Now we consider the superconvergence at nodes.
   In light of \eqref{eq:3} and \eqref{PG:error},  we obtain
\begin{eqnarray*}
   |e_h(x_i)|&=&\left| a(e_h,G(x_i,\cdot))\right|=\left| (-\alpha e_h''+\beta e_h'+\gamma e_h, G(x_i,\cdot)-{\cal I}_{k-2}G(x_i,\cdot))\right|\\
   &\lesssim & h^{2(k-1)}|u|_{k+1,\infty},
\end{eqnarray*}
  where in the last step, we have used  the fact the Green function
  $G(x_i,\cdot)\in C^{k}(\tau_j), j\in\bZ_N$ is bounded.
   Similarly, we have from \eqref{eh} and \eqref{optimal:1}
\[
   |e'_h(x_i)|\lesssim h^{k-1}\|e_h\|_{2} \lesssim h^{2(k-1)}|u|_{k+1,\infty}.
\]
  Then \eqref{super:nodes} follows.

 We next prove \eqref{approx:super1:pg}-\eqref{super:secondderivative}. We consider two cases, i.e., $k\ge 4$ and $k=3$.

  {\bf Case 1: $k\ge 4$}

 For any function $v\in L^2(\Omega)$, we denote by ${\cal I}_{k-2} v$ the $L^2$ projection of $v$ onto $P_{k-2}$ and define
 \[
    \partial_x^{-1} v(x):=\int_{a}^x  v(t) dt.
 \]
 In light of \eqref{eq:3}, we have from the integration by parts,  the orthogonality \eqref{ortho:3} and  \eqref{PG:error},
\begin{eqnarray*}
   \xi (x)=a(\xi,G_h)&=&( -\alpha  \xi''+\beta  \xi'+\gamma  \xi,G_h)=( -\alpha  \xi'', {\cal I}_{k-2} G_h)+(\beta \xi'+\gamma \xi, G_h)\\
   &=&( -\alpha  e_h'', {\cal I}_{k-2} G_h)+(\beta \xi'+\gamma \xi, G_h)\\
   &=&-(\beta e_h'+\gamma e_h, {\cal I}_{k-2} G_h)+(\beta \xi'+\gamma \xi, G_h)\\
   &=&(\beta e_h'+\gamma e_h,  G_h-{\cal I}_{k-2} G_h)-(\beta\eta'+\gamma\eta, G_h)=I_1-I_2.
\end{eqnarray*}
  Now we estimate the two terms $I_1, I_2$, respectively.
 In light of \eqref{optimal:1} and the fact that $\|G_h\|_{2,1}\lesssim 1$ (see, e.g., \cite{Chen.C.M2001}), we have
\[
  |I_1|=|(\beta e_h'+\gamma e_h,  G_h-{\cal I}_{k-2} G_h)|\lesssim h^2(\|e_h\|_{1,\infty}+\|e_h\|_{0,\infty})\|G_h\|_{2,1}\lesssim h^{k+2}|u|_{k+1,\infty}.
\]
   On the other hand, by \eqref{ortho:3}, there  holds for $k\ge 4$,
\[
   \int_{\tau_i}(u-u_I)(x) dx=0,
\]
    and thus
\[
   (\partial_x^{-1} \eta)(x_i)=0,\ \ i\in\bZ_N,\ \ (\partial_x^{-1} \eta)(x)=\int_{x_{i-1}}^x \eta(t) dt,\ \ \forall x\in\tau_i.
\]
  Then by the integration by parts,
\begin{eqnarray*}
 |I_2|&=&|(\beta\eta'+\gamma\eta, G_h)|=|(\beta \partial_x^{-1}\eta,G_h'')-(\gamma \partial_x^{-1}\eta,G_h')|\\
 &\lesssim &\| \partial_x^{-1}\eta\|_{0,\infty}\|G_h\|_{2,1}\lesssim h\|\eta\|_{0,\infty}.
\end{eqnarray*}
  Consequently,
\[
    | \xi (x)|\le |I_1|+|I_2|\lesssim h^{k+2}|u|_{k+1,\infty},\ \ \forall x\in\Omega,
\]
  which yields, together with the inverse inequality,
\[
    \| \xi \|_{0,\infty}\lesssim h^{k+2}|u|_{k+1,\infty},\ \  \| \xi \|_{1,\infty}\lesssim h^{k+1}|u|_{k+1,\infty},\ \ \| \xi \|_{2,\infty}\lesssim h^{k}|u|_{k+1,\infty}.
\]
  Then the desired results \eqref{approx:super1:pg}-\eqref{super:secondderivative} follow from the triangle inequality and
  the approximation properties of $u_I$ in  Theorem \ref{theo:1}.

  {\bf Case 2: $k=3$}

   To  prove \eqref{super:derivlobatto} and \eqref {super:secondderivative} for $k=3$,
 % For any function $v\in V_h$,  noticing that $v\in C^1(\Omega)\cap P_3$,  we need $2N$ conditions to determine the function $v$.
   we  first construct a special function $w_h\in P_3\cap C^1(\Omega)$  satisfying the following condition:
\begin{eqnarray}\label{w}
   &&(\alpha w_h'',v)=(\beta\eta'+\gamma\eta, v)\ \ \forall v\in P_1(\tau_j)\setminus P_0(\tau_j),\\\label{w1}
      && w_h'(x_i)=0,\ \ w_h(a)=0\ \ \forall i=0,\ldots, N.
\end{eqnarray}
  We can  prove that the function $w_h$ is uniquely defined. Actually, if the right hand side of \eqref{w}
  equals to zero, we can easily obtain that $w_h''=0$. Then the boundary condition \eqref{w1} indicates that $w_h=0$.
  We next estimate the function $w_h$.
  We suppose
\[
    w_h(x)=\sum_{i=1}^Nc_i\phi_i(x)
\]
  with $\phi_i(x)\in P_3\cap C^1(\Omega)$ being the basis function associated with the node $x_i$, that is,
\[
 \phi_i(x)=\left\{\begin{array}{lll}
 \frac{1}{h_{i+1}^3}(x_{i+1}-x)^2(2x+x_{i+1}-3x_i), & \text{if} & x\in\tau_{i+1}\\
\frac{1}{h_{i}^3}(x-x_{i-1})^2(3x_i-2x-x_{i-1}), & \text{if} & x\in\tau_i,\\
0,& \text{else} &
\end{array}
\right.
\]
  We choose $v=x$ in \eqref{w} to obtain
\[
    \frac{12(c_j-c_{j-1})}{h_j^3}(\bar x_j-x,x)_j=(\beta\eta'+\gamma\eta, x)_j,
\]
  where $(w_h,v)_j=\int_{\tau_j} w_hv dx, \bar x_j=\frac{x_j+x_{j+1}}{2}$. Consequently,
\[
   |c_j-c_{j-1}|\lesssim h_j^{k+3}|u|_{k+1,\infty},
\]
   and thus,
 \[
    \|w_h''\|_{0,\infty,\tau_j}\lesssim \frac{|c_j-c_{j-1}|}{h_j^3}\lesssim h^k|u|_{k+1,\infty}.
 \]
   Moreover,  there hods for all $x\in\tau_j$
 \[
    |w_h'(x)|=\left|w_h'(x_{j-1})+\int_{x_{j-1}}^x w_h''(x) dx\right|\le h_j\|w_h''\|_{0,\infty,\tau_j}\lesssim h^{k+1}|u|_{k+1,\infty}.
 \]
  Then
 \[
    \|w'_h\|_{0,\infty}\lesssim h^{k+1}|u|_{k+1,\infty},\ \ \|w_h\|_{0,\infty}\lesssim \|w'_h\|_{0,\infty}\lesssim h^{k+1}|u|_{k+1,\infty}.
 \]

   Now we are ready to prove \eqref{super:derivlobatto} and \eqref {super:secondderivative} for $k=3$.
    Let $$e_h=u-u_h=\tilde\xi+\tilde \eta,\ \ \ \tilde \xi:=u_I-u_h-w_h,\ \ \ \tilde \eta:=u-u_I+w_h.$$
     Choosing $v_h=-\tilde \xi'' $ in \eqref{PG:error} following the same arguments as that in \eqref{eq:5}, we obtain
 \begin{eqnarray*}
     (\alpha \tilde \xi'', \tilde\xi'')+(\gamma \tilde\xi',\tilde\xi')-\frac{\beta}{2}(|\tilde\xi'(b)|^2-|\tilde\xi'(a)|^2)&=&(-\alpha \tilde\eta''+\beta\tilde \eta'+\gamma \tilde\eta,\tilde\xi'')\\
        &=& (\beta  \eta'+\gamma  \eta,\tilde\xi'')+(-\alpha w_h''+\beta w_h'+\gamma w_h,\tilde\xi'')=I.
 \end{eqnarray*}
   We now estimate the term $I$.  Since $\tilde \xi''|_{\tau_j}\in P^1(\tau_j)$, we have the following decomposition
\[
    \tilde\xi''=\xi_0+\xi_1,\ \ \xi_0\in P^0(\tau_j),\ \ \xi_1\in P^1(\tau_j)\setminus P_0(\tau_j).
\]
  By \eqref{w} and the integration by parts, we get
\begin{eqnarray*}
|I|
&=&\left|(\beta  \eta'+\gamma  \eta,\xi_0)+(-\alpha w_h''+\beta w_h'+\gamma w_h,\xi_0)+(\beta w_h'+\gamma w_h,\xi_1)\right|\\
&=&\left|(\eta,\xi_0)+(\beta w'+\gamma w_h,\xi_0)+(\beta w_h'+\gamma w_h,\xi_1)\right|\lesssim (\|\eta\|_{0}+\|w_h\|_{1})\|\tilde\xi''\|_{0}.
\end{eqnarray*}
   In light of the estimates for $w_h$ and $\eta$, we get
\begin{eqnarray}\label{eq:8}
    (\alpha \tilde \xi'', \tilde\xi'')+(\gamma \tilde\xi',\tilde\xi')-\frac{\beta}{2}(|\tilde\xi'(b)|^2-|\tilde\xi'(a)|^2)\lesssim h^{k+1}|u|_{k+1,\infty}|\tilde\xi|_{2}.
\end{eqnarray}
 By \eqref{super:nodes}, there holds
\begin{eqnarray*}
   |\tilde\xi'(b)|=|\xi'(b)|\lesssim h^{2(k-1)}|u|_{k+1,\infty}.
\end{eqnarray*}
   Substituting the above estimate into \eqref{eq:8} and  using the Cauchy-Schwartz inequality yields
\[
   \|\tilde \xi''\|_0\lesssim h^{k+1}|u|_{k+1,\infty}.
\]
  By the triangle inequality and the inverse inequality,
\[
   \| \xi''\|_{0,\infty}\le \| \tilde \xi''\|_{0,\infty}+\|w_h''\|_{0,\infty}\lesssim h^{-\frac 12}\| \tilde \xi''\|_{0}+h^{k+1}|u|_{k+1,\infty}
   \lesssim h^{k}|u|_{k+1,\infty}.
\]
  Furthermore, there holds for all $x\in\tau_j$
\begin{eqnarray*}
   |\xi'(x)|=\left|\xi'(x_{j-1})+\int_{x_{j-1}}^x \xi''(x) dx\right|&\lesssim & h^{2(k-1)}|u|_{k+1,\infty}+h\| \xi''\|_{0,\infty}\\
   &\lesssim & h^{2(k-1)}|u|_{k+1,\infty}+h^{k+1} |u|_{k+1,\infty}.
\end{eqnarray*}
  Then  \eqref{super:derivlobatto} and \eqref {super:secondderivative} follows from the triangle inequality and
  the approximation properties of $u_I$  for $k=3$.
  This finishes our proof.
\end{proof}

\begin{remark}
 As we may observe from the above theorem, for problems with constant coefficients,
 the convergence rate of the  error $\|u_h-u_I\|_2$ is two order higher than the optimal convergence rate $k-1$ in case of $\beta=0,\gamma\neq 0$.
  However, this superconvergence result may not hold true for problems with variable coefficients.
  Actually, in case of $\beta=0, \alpha\neq 0, \gamma\neq 0$ with
  $\alpha $ a  variable function, we have from \eqref{eq:7}
  \begin{eqnarray*}
     (\alpha \xi'', \xi'')+(\gamma\xi',\xi')-\frac{\beta}{2}(|\xi'(b)|^2-|\xi'(a)|^2)&=&(-\alpha \eta''+\beta\eta'+\gamma \eta,\xi'')\\
        &=& ((\bar \alpha- \alpha) \eta''+\beta\eta'+\gamma \eta,\xi'')\\
       &\lesssim & h^{k} \|\xi''\|_0|u|_{k+1,\infty},
 \end{eqnarray*}
   where $\bar\alpha$ denotes the cell average of $\alpha$, i.e., $\bar\alpha|_{\tau_j}=h_j^{-1}\int_{\tau_j}\alpha (x) dx$.
   Then we follow the same argument as that in Lemma \ref{lemma:1} to obtain
 \[
    \|\xi''\|\lesssim h^{k}|u|_{k+1,\infty}.
 \]
   In other words, the convergence rate of $\|u_h-u_I\|_2$
   for problems with variable coefficients is always $k$, only one order higher than the optimal convergence rate.
  This is the difference between the constant coefficients and variable coefficients.
  Our numerical examples will demonstrate this point.
\end{remark}

\section{Superconvergence for Gauss Collocation methods}

  This section is dedicated to the superconvergence analysis of the Gauss collocation method.
  Our analysis is along this line: we first prove that the Gauss collocation solution $\bar u_h$ is superclose to
    the Petrov-Galerkin solution $u_h$; then due to the supercloseness between $\bar u_h$ and $u_h$,
   the numerical solution $\bar u_h$ shares the same superconvergence results with that of $u_h$, and finally we
   establish all  superconvergence results for the solution of the Gauss collocation method.

  We begin with some preliminaries. We first denote by $\omega_{im}, (i,m)\in\bZ_{N}\times \bZ_{k-1}$ the wight of Gauss quadrature.
  For any function $u,v$, we define the following discrete $L^2$ inner product $(\cdot,\cdot)_*$ as
\[
    (u,v)_*:=\sum_{i=1}^N\sum_{m=1}^{k-1}(uv)(g_{im})\omega_{im}.
\]
  For any $v_h\in W_h$, we multiply $v_h(g_{im})\omega_{im}, (i,m)\in\bZ_{N}\times \bZ_{k-1}$ on both sides of  \eqref{collocation} and sum up all $m$ from $1$ to $k-1$  to derive
\begin{equation}\label{eq:1}
    \sum_{m=1}^{k-1}(-\alpha  \bar  u_h''+\beta \bar  u'_h+\gamma \bar  u_h)(g_{im})v_h(g_{im})\omega_{im}=\sum_{m=1}^{k-1}f(g_{im})v_h(g_{im})\omega_{im}.
\end{equation}
  As we may observe, the $C^1$ Gauss collocation method can be viewed as the counterpart Petrov-Galerkin method up to  a Gauss numerical integration error.
  Note that the $(k-1)$-point Gauss quadrature is exact for all  polynomials of degree not less than $2k-3$. Then
 \begin{equation}\label{bilinear:1}
   a_*(\bar u_h,v_h):=  (-\alpha  \bar u_h'', v_h)+(\beta \bar  u_h' ,v_h)+(\gamma \bar  u_h,v_h)_{*}=(f,v_h)_*,\ \ \forall v_h\in W_h.
 \end{equation}
   Denote
 \[
    \bar e_h=u_h-\bar u_h.
 \]
    Subtracting  \eqref{bilinear:1} from \eqref{PG}, we have
 \begin{equation}\label{bilinear:2}
    (-\alpha  \bar  e_h'', v_h)+(\beta \bar  e_h' ,v_h)-(\gamma  \bar u_h,v_h)_{*}+(\gamma u_h,v_h)=(f,v_h)-(f,v_h)_*,\ \ \forall v_h\in W_h,
 \end{equation}
  or equivalently,
\begin{equation}\label{bilinear:3}
    (-\alpha  \bar  e_h''+\beta \bar  e_h'+\gamma \bar e_h ,v_h)=(\gamma  \bar  u_h,v_h)_{*}-(\gamma \bar u_h,v_h)+(f,v_h)-(f,v_h)_*,\ \ \forall v_h\in W_h.
 \end{equation}
   We note that up to a Gauss numerical quadrature error, the right hand side of the above equation equals to zero.

  We have the following supercloseness result for the error $\bar e_h$.

 \begin{theorem}\label{theo:2}
 Assume that $u\in W^{2k,\infty}(\Omega)$ is the solution of \eqref{con_laws}, and
 $u_h$ and $\bar u_h$ is the solution of \eqref{PG} and \eqref{collocation}, respectively. Then
\begin{equation}\label{supercloseness:1}
      \|u_h-\bar u_h\|_{0,\infty}+h\|u_h-\bar u_h\|_{1,\infty}+h^2\|u_h-\bar u_h\|_{2}\lesssim h^{k+2} \|u\|_{2k,\infty}.
\end{equation}
\end{theorem}
\begin{proof} Noticing that $\bar e_h\in V^0_h$,  we  choose $v_h=-\bar e_h''\in W_h$ in \eqref{bilinear:2} and use the integration by parts to obtain
  \begin{eqnarray}\label{eq:10}
       (\alpha\bar e_h'', \bar e_h'')+(\gamma\bar e_h',\bar e_h')_* -\frac{\beta}{2}(|\bar e_h'(b)|^2-|\bar e_h '(a)|^2)= (f-\gamma u_h,\bar e_h'')_*-(f-\gamma u_h, \bar e_h'').
      %     &=&(f,(u_h-u_I)'')_*-a_*(u_I, (u_h-u_I)'')+(\gamma (u_h-u_I),(u_h-u_I)'')_{*}\\
%           &=& I_1+I_2+I_3+ I_4,
  \end{eqnarray}
   For any function $w$,   we denote ${\cal I}_hw\in P_{k-1}$ the Gauss interpolation function of
   $w$ satisfying
 \[
    {\cal I}_hw(x_i)=w(x_i),\ \  {\cal I}_hw(g_{im})=w(g_{im}),\ \ m\in \bZ_{k-1}.
 \]
  Since  $v_h {\cal I}_hw\in P_{2k-3}$ for all  $v_h\in W_h$,  we have
 \begin{equation}\label{eq:9}
     |(w,v_h)-(w,v_h)_*|=|(w-{\cal I}_hw,v_h)|\lesssim h^{k}\|w\|_{k}\|v_h\|_0,\ \ \forall v_h\in W_h.
 \end{equation}
   Plugging the above estimate into \eqref{eq:10} gives
\begin{eqnarray}\label{eq:7}
\begin{split}
    (\alpha\bar e_h'', \bar e_h'')+(\gamma\bar e_h',\bar e_h')_*&\lesssim  h^{k} (|f|_{k}+|u_h|_{k})\|\bar e_h''\|_0+|\bar e_h'(b)|^2&\\
      &\lesssim  h^{k}(|f|_k+|u|_{k+1,\infty}) \|\bar e_h''\|_0+|\bar e_h'(b)|^2,&
\end{split}
\end{eqnarray}
  where in the last step, we have used \eqref{optimal:1}, the inverse inequality and the triangle inequality to get
\begin{eqnarray}\label{eq:11}
\begin{split}
    |u_h|_{k}&\lesssim |u_I|_k+h^{-k} |u_I-u_h|_{0}&\\
    &\lesssim  |u|_k+|u-u_I|_k+h^{-k}(\|u_I-u\|_{0}+\|u_h-u\|_{0})\lesssim |u|_{k+1,\infty}.&
\end{split}
\end{eqnarray}
%
%\begin{eqnarray*}
%    I_2=(\gamma (u-u_I),(u_h-u_I)'')\lesssim h^{k} \|u\|_{k+1}\|(u_h-u_I)''\|,
%\end{eqnarray*}
%  Similarly,
% \begin{eqnarray*}
%     I_3=(\gamma u_I, (u_h-u_I)'')-(\gamma u_I, (u_h-u_I)'')_*\lesssim h^{k}  \|u\|_{k}\|(u_h-u_I)''\|,
% \end{eqnarray*}
%  By Cauchy-Schwartz inequality, we have
%\[
%I_4\lesssim \|u_h-u_I\| \|(u_h-u_I)''\|\lesssim h^{k} \|u\|_{k+1}\|(u_h-u_I)''\|
%\]
 % Using the error of Gauss quadrature, there exists $\xi_i$ such that
%  \begin{eqnarray*}
%     I_3 & = &\sum_{i=1}^{N} \frac{h_i^{2k-1}}{(k-1)!}{(2k-1)[(2(k-1))!]^3} (u_I(u_h-u_I)'')^{(2k-2)}(\xi_i) \\
 %       & \lesssim &\sum_{i=1}^{N} h_i^{2k-1}  \|u_I\|_{k+1,\infty}\|(u_h-u_I)''\|_{k-2,\infty} \\
 %       & \lesssim &
% \end{eqnarray*}
 To estimate $\bar e_h'(b)$, we choose $v=\bar e'_h$ in \eqref{eq:3} and using the integration by parts to obtain
\begin{eqnarray*}
  \bar e'_h(x_i)&=&a(\bar e'_h,G(x_i,\cdot))=(\alpha \bar e_h''-\beta \bar e_h',G'(x_i,\cdot))
   +(\gamma \bar e'_h,G(x_i,\cdot)) \\
  &=& (\alpha \bar e_h''-\beta \bar e_h'-\gamma \bar e_h, G'(x_i,\cdot)-\bar G')+(\alpha \bar e_h''-\beta \bar e_h'-\gamma \bar e_h, \bar G')\\
  &=&(\alpha \bar e_h''-\beta \bar e_h'-\gamma \bar e_h, G'(x_i,\cdot)-\bar G')-(f,\bar G')+(f,\bar G')_*,
 %  &\lesssim & h^{2(k-1)}\|u\|_{k+1,\infty}.
\end{eqnarray*}
  where $\bar G'\in P_0$  denotes the cell average of $G'(x_i,\cdot)$
  and in the last step, we have used \eqref{bilinear:3} and the fact that
\[
   (\bar u_h,v)-(\bar u_h,v)_*=0,\ \ \forall v\in P_{k-3}.
\]
  Using the
  fact that $G(x_i,\cdot)\in C^{k}(\tau_j)$ is bounded, we get
 \[
    (\alpha \bar e_h''-\beta \bar e_h'-\gamma \bar e_h, G'(x_i,\cdot)-\bar G')\lesssim h\|\bar e_h\|_{2}.
 \]
   On the other hand, by \eqref{eq:9}, we have
\begin{eqnarray*}
 |(f,\bar G')-(f,\bar G')_*|\lesssim h^{k}|f|_{k}.
\end{eqnarray*}
  Consequently,
\[
    |\bar e'_h(x_i)|\lesssim h\|\bar e_h\|_{2}+h^{k}|f|_{k}.
\]
 Substituting the above inequality into \eqref{eq:7} and using the Cauchy-Schwartz inequality yields
\[
  (\alpha\bar e_h'', \bar e_h'')+(\gamma\bar e_h',\bar e_h')_* \le ( \frac{\alpha}{4}+Ch^2) \|\bar e_h''\|^2_0+C_1h^{2k}(|f|_{k}+|u|_{k+1,\infty})^2
\]
   for some positive $C,C_1$. Therefore, when $h$ is sufficient small, there holds
\[
   \|\bar e_h''\|_0\lesssim h^k (|u|_{k+1,\infty}+ |f|_k)\lesssim h^{k}\|u\|_{k+2,\infty}.
\]

  We next estimate $\|\bar e_h\|_{0,\infty}$.
Choosing $v=\bar e_h$ in \eqref{eq:4} and using \eqref{bilinear:3}, we get
\begin{eqnarray}\nonumber
     \bar e_h(x)&=&a(\bar e_h,G_h)=(-\alpha \bar e_h''+\beta \bar e_h'+\gamma \bar e_h, G_h)\\\nonumber
       &=&(-\alpha \bar e_h''+\beta \bar e_h'+\gamma \bar e_h, G_h-{\cal I}_{k-2}G_h)+(-\alpha \bar e_h''+\beta \bar e_h'+\gamma \bar e_h,{\cal I}_{k-2}G_h)\\\label{error:function}
       &=&(\beta \bar e_h'+\gamma \bar e_h, G_h-{\cal I}_{k-2} G_h)+(f-\gamma \bar u_h, {\cal I}_{k-2}G_h)-(f-\gamma \bar u_h,{\cal I}_{k-2}G_h)_*.
 \end{eqnarray}
   Here again ${\cal I}_{k-2} G_h$ denotes the $L^2$ projection of $G_h$ onto $P_{k-2}$.
   By using  the error of Gauss quadrature (see, e.g., \cite{DavisRabinowitz1984}, P.98 (2.7.12)),   there exists some $\theta_j\in\tau_j$ such that
 \begin{eqnarray*}
    (f,{\cal I}_{k-2}G_h)-(f,{\cal I}_{k-2}G_h)_*&=&\sum_{j=1}^N\frac{h^{2k-1}[(k-1)!]^4}{(2k-1)[(2k-2)!]^3}(f{\cal I}_{k-2}G_h)^{(2k-2)}(\theta_j)\\
      &\lesssim & h^{2k-1}\|f\|_{2k-2,\infty}\sum_{j=1}^N\|{\cal I}_{k-2}G_h\|_{k-2,\infty,\tau_j}\\
      &\lesssim & h^{k+2}\|f\|_{2k-2,\infty}\|G_h\|_{2,1}.
 \end{eqnarray*}
   Here in the last step, we have used the inverse inequality
\[
   \|v_h\|_{m,p}\lesssim h^{n-m+\frac{1}{p}-\frac{1}{q}} \|v_h\|_{n,q},\ \ \forall n<m.
\]
   Similarly,  there holds
\begin{eqnarray*}
  (\gamma \bar u_h,{\cal I}_{k-2}G_h)-(\gamma \bar u_h,{\cal I}_{k-2}G_h)_*&= &\sum_{j=1}^N\frac{h^{2k-1}[(k-1)!]^4}{(2k-1)[(2k-2)!]^3}(\bar u_h{\cal I}_{k-2}G_h)^{(2k-2)}(\theta_j)\\
  &\lesssim & h^{2k-1}\|\bar u_h\|_{k,\infty}\|G_h\|_{k-2,\infty}
  \lesssim  h^{k+2}\| \bar u_h\|_{k,\infty}\|G_h\|_{2,1}\\
  &\lesssim &  h^{k+2}(\| \bar e_h\|_{k,\infty}+\|u_h\|_{k,\infty})\|G_h\|_{2,1}.
 % &\lesssim & h^{k+2}\|u\|_{k,\infty}+h^{k+2}h^{\frac 32-k}\|e_h\|_{2}\lesssim h^{k+2}\|u\|_{k,\infty}.
\end{eqnarray*}
   On the other hand,
\begin{eqnarray*}
   |(\beta \bar e_h'+\gamma \bar e_h, G_h-{\cal I}_{k-2} G_h)|\lesssim h^2\|G_h\|_{2,1}\|\bar e_h\|_{1,\infty}\lesssim
   h\|\bar e_h\|_{0,\infty}\|G_h\|_{2,1}.
\end{eqnarray*}
 Since $\|G_h\|_{2,1}$ is bounded, we have
\begin{eqnarray*}
   |\bar e_h(x)|&\lesssim & h^{k+2}\|f\|_{2k-2,\infty}+h\|\bar e_h\|_{0,\infty}+h^{k+2}(\| \bar e_h\|_{k,\infty}+\|u_h\|_{k,\infty})\\
     &\lesssim &h^{k+2}\|u\|_{2k,\infty}+h\|\bar e_h\|_{0,\infty}.
\end{eqnarray*}
   Here in the last step, we have used \eqref{eq:11} and  the inverse inequality $\|\bar e_h\|_{k,\infty}\lesssim h^{-k}\|\bar e_h\|_{0,\infty}$.
   Consequently,
\[
    \|\bar e_h\|_{0,\infty}\lesssim h^{k+2}\|u\|_{2k,\infty},\ \ \|\bar e_h\|_{1,\infty}\lesssim h^{-1}\|\bar e_h\|_{0,\infty}\lesssim
    h^{k+1}\|u\|_{2k,\infty}.
\]
   This finishes our proof. $\Box$
%\begin{eqnarray*}
%    (\alpha \theta',\theta')+(\gamma \theta,\theta)_* &\le & \sum_{i=1}^N \sum_{m=1}^{k-1}(\alpha (u-u_I)''-\beta (u-u_I)'+\gamma (u-u_I))(g_{im})\theta(g_{im})\omega_m\\
%      &\le &  h^{2k} +\frac{1}{2}(\gamma \theta,\theta)_*.
%\end{eqnarray*}
%  Consequently,
% \[
%    \|\eta\|_{1}\lesssim h^{k}.
% \]
\end{proof}

Using the conclusions in the above theorem and the superconvergence results for the Petrov-Galerkin method, we have the
following superconvergence properties for the solution of Gauss collocation methods.

\begin{theorem}\label{theo:3}
 Assume that $u\in W^{2k,\infty}(\Omega)$ is the solution of \eqref{con_laws}, and $\bar u_h$ is the solution of \eqref{collocation}.
The following superconvergence properties hold true.
  \begin{enumerate}
  \item Superconvergence of function value approximation on interior roots of  $J_{k+1}^{-2,-2}$ for $k\ge 4$:
  \begin{equation}\label{approx:super1}
  |(u-\bar u_h) (l_{im}) | \lesssim   h^{k+2}\|u\|_{2k,\infty},
\end{equation}
where $l_{im}$, $m=1,\cdots, k-3$ are interior roots of $\hat J_{k+1}^{-2,-2}(x)$ in $\tau_i$.
 \item
 Superconvergence of first order derivative value approximation on  Gauss-Lobatto points:
\begin{equation}
   |(u-\bar u_h)'(gl_{in})| \lesssim h^{k+1}\|u\|_{2k,\infty},
\end{equation}
where
$gl_{in}, i\le k-2$ are interior Gauss-Lobatto points of degree $k-2$.

\item Superconvergence of second order derivative value approximation on  Gauss   points:
\begin{equation}
   |(u-\bar u_h)''(g_{in})| \lesssim h^{k}\|u\|_{2k,\infty},
\end{equation}
where  $g_{in}$, $n\le k-1$ are interior roots of $L_{k-1}$, i.e.,
 the  Gauss  points of degree $k-1$.

\item  Supercloseness between the numerical solution and the truncation projection of the exact solution in the $H^2$ norm:
\begin{equation}\label{collocation:super}
    \|\bar u_h-u_I\|_2\lesssim h^{k}\|u\|_{2k,\infty}.
\end{equation}

\item Superconvergence for both function value and derivative value approximations at nodes:
\begin{equation}\label{super-collocation:nodes}
   |(u-\bar u_h)(x_i)|\lesssim h^{2(k-1)}\|u\|_{2k,\infty},\ \ |(u-\bar u_h)'(x_i)|\lesssim h^{2(k-1)}\|u\|_{2k,\infty},\ \ i\in\bZ_N.
\end{equation}
\end{enumerate}
 \end {theorem}
\begin{proof} We only prove \eqref{super-collocation:nodes} since \eqref{approx:super1}-\eqref{collocation:super} follow directly from
Theorem \ref{theo:1} and Theorem \ref{theo:2}.
 In light of \eqref{error:function}, we have
\begin{eqnarray*}
   \left|\bar e_h(x_i)\right|&=&\left|(-\alpha \bar e_h''+\beta \bar e_h'+\gamma \bar e_h, G_h(x_i,\cdot)-{\cal I}_{k-2}G_h(x_i,\cdot))+I\right|\\
   &\lesssim & h^{k-1}\|e_h\|_{2}+|I|,
\end{eqnarray*}
  where we used that $G_h\in C^{k}(\tau_j)$ and
\[
   I=(\gamma  \bar  u_h-f,{\cal I}_{k-2}G(x_i,\cdot))_{*}+(\gamma \bar u_h-f,{\cal I}_{k-2}G(x_i,\cdot)).
\]
  Again we use the error of Gauss numerical quadrature and the fact $\|{\cal I}_{k-2}G_h\|_{k-2,\infty}\lesssim 1$
  to get
\[
    |I|\lesssim h^{2k-2}(\|f\|_{2k-2,\infty}+\|\bar  u_h\|_{k,\infty})\|{\cal I}_{k-2}G_h\|_{k-2,\infty}\lesssim
    h^{2k-2}\|u\|_{2k,\infty},
\]
  and thus
 \[
  \left|\bar e_h(x_i)\right| \lesssim h^{2k-2}\|u\|_{2k,\infty}.
 \]
  Similarly, we can prove
 \begin{eqnarray*}
   |\bar e'_h(x_i)|\lesssim h^{2(k-1)}(\|u\|_{k+1,\infty}+\|f\|_{2k-2,\infty})\lesssim h^{2k-2}\|u\|_{2k,\infty}.
\end{eqnarray*}
  Then the proof is complete.
\end{proof}

\begin{remark}
   As we may observe from Theorem \ref{theo:1} and Theorem \ref{theo:3}, to achieve the same
   superconvergence result,
   the regularity assumption of the exact solution $u$ for the Gauss collocation method
   is much more stronger than that for the counterpart Petrov-Galerkin method.
\end{remark}

\section{Numerical experiments}

  In this section, we present some numerical examples to demonstrate
  the method and to verify the theoretical findings established in previous sections.

  In our numerical experiments, we solve the model problem \eqref{con_laws}
  by the $C^1$ Petrov-Galerkin method \eqref{PG} and the Gauss collocation method \eqref{collocation} with $k=3$ and $k=4$.
  We test various errors in our examples, including the $H^2$ error of  $u_h-u_I$ denoted as $\|u_h-u_I\|_2$,
  the maximum errors of $u-u_h$ and $(u-u_h)'$ at mesh points,  the maximum errors of $u-u_h$
  at interior roots of $\hat J_{k+1}^{-2,-2}(x)$,
   $(u-u_h)'$ and $(u-u_h)''$ at interior Gauss-Lobatto and Gauss  points, respectively.   They are defined by
\begin{align*}
&e_{un}=\max_{i} |(u-u_h)(x_i)|, \quad e_{u'n}=\max_{i} |(u-u_h)'(x_i)|,\\
&e_u=\max_{i,m} |(u-u_h)(l_{im})|, \quad e_{u'}=\max_{i,n} |(u-u_h)'(gl_{in})|, \quad e_{u''}=\max_{i,n} |(u-u_h)''(g_{in})|.
\end{align*}
Here $l_{im}, 1\leq m\leq k-3$ are interior roots of $\hat J_{k+1}^{-2,-2}(x)$, and
$gl_{in}, 1\leq n\leq k-2$ are interior Lobatto points, and
$g_{in}, 1\leq n\leq k-1$ are interior Gauss points in $\tau_i$.
For simplicity, we do not distinguish the error symbols when the method is clearly stated in the following tables.

{\it{Example 1.}} We consider the following equation with Dirichlet boundary condition:
\begin{equation}\label{model}
\begin{cases}
-\alpha u''(x) + \beta u'(x) +\gamma u(x) = f(x), \quad x\in[0,1],\\
u(0) = u(1) = 0.
\end{cases}
\end{equation}

 We take the constant coefficients as
\[
 \alpha=\beta=\gamma=1,
\]
  and choose the right-hand side function $f$ such that the exact solution to this problem is
$$u(x) = \sin(\pi x).$$

 Non-uniform meshes of $N$ elements  are used in our numerical experiments with $N=2,\ldots,32$, which are obtained
 by randomly and independently perturbing each
node of a uniform mesh by up to some percentage.  To be more precise,
%We obtain the non-uniform meshes by first equally dividing the interval $[0,1]$ into $N$ subintervals with $N=2, ... , 32$ and then giving the mesh nodes a small random perturbation, i.e.,
$$x_j = \frac{j}{N} + 0.01\frac{1}{N} \sin(\frac{j\pi}{N})\text{randn}(),\quad 0\leq j\leq N,$$
where $randn()$ returns a uniformly distributed random number in $(0,1).$

We list in Table \ref{C1PG_exam1} various approximation errors calculated by the
$C^1$ Petrov-Galerkin method for $k=3,4$. As we may observe, both the convergence rates of the error $e_{un}$ and $e_{u'n}$
are $2k-2$, and the convergence rate of $e_{u},e_{u'},e_{u''}$ is $k+2,k+1,k$, respectively. All these results are consistent
with our theoretical findings in Theorem \ref{theo:1}.

\begin{table}[H]
   \caption{\small{Errors, corresponding convergence rates for $C^1$ Petrov-Galerkin method,  $\alpha=\beta=\gamma=1$.}}\label{C1PG_exam1}
  \vspace{0.1cm}
   \centerline{
   {\footnotesize
  \begin{tabular}{cccccccccccc}
  \hline
  \quad& \quad& \multicolumn{2}{c}{$e_{un}$}&\multicolumn{2}{c}{$e_{u'n}$} &
  \multicolumn{2}{c}{$e_u$}& \multicolumn{2}{c}{$e_{u'}$} & \multicolumn{2}{c}{$e_{u''}$}\\
  \hline
  $k$ & $N$ & error & order& error & order& error & order & error & order & error & order\\
  \hline
  \quad & 2&  8.03e-04 & - &6.11e-03 & -& - & - & 7.59e-03 & - & 1.31e-01 & -\\
  \quad & 4 & 7.02e-05 & 3.49  & 4.92e-04 & 3.61 & - & - & 5.61e-04 & 3.73  & 1.66e-02	&2.98 \\
        3 & 8 & 4.59e-06 & 3.95  & 3.02e-05 & 4.04 &- & - & 3.73e-05 & 3.92  & 2.18e-03	&2.93 \\
  \quad & 16 &2.91e-07 & 4.04  & 1.90e-06 & 4.05 &- & - & 2.66e-06	&3.87  & 2.67e-04	&3.03 \\
    \quad & 32 &1.80e-08	& 4.00  & 1.18e-07 & 3.99  & - & - & 1.77e-07 & 3.90  & 3.38e-05&2.98 \\
  \hline \hline
    \quad & 2&2.88e-05 & - & 2.23e-05 & - & 7.54e-05& - &8.72e-04 & - & 1.30e-02 & -\\
  \quad & 4 & 4.25e-07 & 6.10  & 2.47e-07 & 6.51  & 1.36e-06 & 5.81&2.44e-05&5.17 	&8.88e-04 & 3.88 \\
        4 & 8 & 6.53e-09 & 6.21  & 5.38e-09 & 5.69  & 2.14e-08 & 6.17 	& 7.91e-07 & 5.10 	& 5.47e-05 & 4.02 \\
  \quad & 16 &1.04e-10 & 5.96  & 1.04e-10 & 5.67  & 3.45e-10 & 5.94& 2.64e-08& 4.89 	&3.73e-06 & 3.87 \\
    \quad & 32 & 1.62e-12 & 6.00  & 1.84e-12	&5.83  & 5.50e-12 & 5.97  & 8.62e-10 & 4.94 &2.30e-07 & 4.02 \\
  \hline
    \end{tabular}}}
\end{table}

We next test the superconvergence behavior of the $C^1$ Gauss collocation method.
 We  present in Table \ref{C1collo_exam1}
  the numerical data for various errors and the corresponding convergence rates calculated by the $C^1$ Gauss collocation method.
We observe  a convergence rate of $2k-2$ for $e_{un}$ and
$e_{u'n}$, $k+2$ for $e_u$,   $k+1$ for $e_{u'}$, and  $k$ for $e_{u''}$, which confirms the theory established in Theorem \ref{theo:3}.

\begin{table}[H]
   \caption{\small{Errors, corresponding convergence rates for $C^1$
   Gauss collocation method, $\alpha=\beta=\gamma=1$.}}\label{C1collo_exam1}
   \vspace{0.1cm}
   \centerline{
   {\footnotesize
  \begin{tabular}{cccccccccccc}
  \hline
  \quad& \quad& \multicolumn{2}{c}{$e_{un}$}& \multicolumn{2}{c}{$e_{u'n}$} & \multicolumn{2}{c}{$e_u$} & \multicolumn{2}{c}{$e_{u'}$} & \multicolumn{2}{c}{$e_{u^{''}}$}\\
  \hline
  $k$ & $N$ & error & order & error & order & error & order & error & order & error & order\\
  \hline
  \quad & 2 & 5.25e-03& - & 1.36e-02 & - & - & - & 1.44e-02 & - &8.32e-02 & -\\
  \quad & 4 & 2.88e-04 & 4.13  & 7.26e-04 & 4.18 & - & - &8.35e-04	&4.06 	&1.16e-02&2.84 \\
        3 & 8 & 1.82e-05 & 4.07  & 4.66e-05 & 4.05 & - & - & 5.89e-05 & 3.91  & 1.61e-03 & 2.85 \\
  \quad & 16 &1.16e-06 & 3.94  & 2.91e-06 & 3.96 &- & - & 4.01e-06	&3.84& 	1.91e-04	&3.08 \\
    \quad & 32 & 7.18e-08 & 4.01  & 1.81e-07	&4.01  & - & - & 2.65e-07 & 3.92  & 2.45e-05	&2.96 \\
  \hline\hline
    \quad & 2 &1.32e-05 & - & 1.04e-04& - & 1.98e-04 & - & 9.54e-04& - & 7.12e-03 & -\\
  \quad & 4 & 2.92e-07 & 5.48  & 1.79e-06 & 5.85  & 3.14e-06 & 5.96 	&3.32e-05 & 4.83  & 5.77e-04 & 3.63 \\
        4 & 8 & 4.62e-09 & 5.97  & 2.80e-08 & 5.99  & 5.09e-08 & 5.94  & 1.05e-06	&4.98  & 3.89e-05 & 3.89 \\
  \quad & 16 &7.56e-11 & 6.06  & 4.40e-10 & 6.12 	&8.61e-10	&6.01  & 3.40e-08	&5.04  & 2.46e-06 & 3.98 \\
    \quad & 32 & 1.18e-12&6.08 & 6.87e-12 & 6.08  & 1.28e-11 & 6.15  & 1.05e-09	&5.08  & 1.52e-07 & 4.02 \\
  \hline
    \end{tabular}}}
\end{table}

Furthermore, we also test the supercloseness result between the $C^1$ numerical solution
and the Jacobi truncation projection  of the exact solution under the $H^2$ norm for
two different choices of parameters: i.e., $\beta=0,\gamma\neq 0$ and $\beta\neq 0,\gamma\neq 0$.
Listed in Table \ref{C1uhuIh2_exam1} are the approximation  errors of $\|u_h-u_I\|_2$ and
their  corresponding convergence  rates.
From Table \ref{C1uhuIh2_exam1} we observe that,
for both  Petrov-Galerkin and Gauss collocation methods, the convergence rate of $\|u_h-u_I\|_2$ is $k$ in case of $\beta\neq 0,\gamma\neq 0$.
However, when $\beta=0,\gamma\neq 0$, the convergence rate is
 $k$ for the Gauss collocation method, and $k+1$ for the Petrov-Galerkin method, which
 %Note that in case of $\beta=0,\gamma\neq 0$, the convergence rate of  $\|u_h-u_I\|_2$ for the Petrov-Galerkin method
 is one order higher than that for the Gauss collocation method.
%All those results  are consistent with our theoretical results predicted in \eqref{pg:superclose} and \eqref{collocation:super}.

\begin{table}[H]
   \caption{\small{ $\|u_h-u_I\|_2$ and the corresponding convergence rates, constant coefficients.}}\label{C1uhuIh2_exam1}
   \vspace{0.1cm}
   \centerline{
   {\footnotesize
  \begin{tabular}{cccccccccc}
  \hline
   \multicolumn{10}{c}{$\|u_h-u_I\|_2$} \\
  \hline
   \multicolumn{2}{c}{\quad}& \multicolumn{4}{c}{$C^1$ Petrov-Galerkin}& \multicolumn{4}{c}{$C^1$ Gauss collocation}\\
   \hline
  \multicolumn{2}{c}{\quad}&  \multicolumn{2}{c}{$\alpha=\beta=\gamma=1$} & \multicolumn{2}{c}{$ \alpha=\gamma=1, \beta=0$} &  \multicolumn{2}{c}{$\alpha=\beta=\gamma=1$} & \multicolumn{2}{c}{$ \alpha=\gamma=1, \beta=0$}  \\
  \hline
  $k$ & $N$ & error & order & error & order  & error & order & error & order\\
  \hline
  \quad & 2 & 3.93e-02 & - & 5.57e-03 & - &1.12e-01 & - & 8.32e-02 & - \\
  \quad & 4 &5.15e-03 & 2.91 & 3.59e-04 & 3.94 & 1.36e-02 & 3.00  &1.07e-02&	2.93 \\
        3 & 8 & 6.47e-04 & 3.00  & 2.23e-05 & 3.99  & 1.72e-03 & 3.04  & 1.34e-03	&3.03 \\
  \quad & 16 &8.12e-05 & 3.04 & 1.40e-06 & 4.02  &2.17e-04 & 2.96  & 1.70e-04&	2.96 \\
    \quad & 32 &1.01e-05	&2.99  & 8.75e-08 & 4.01 & 2.70e-05& 3.01  & 2.12e-05	&3.03 \\
  \hline\hline
    \quad & 2 & 3.47e-03 & - & 2.23e-04& - &9.86e-03 & - & 8.27e-03 &-\\
  \quad & 4 & 2.24e-04 & 3.97  & 7.13e-06 & 4.99  & 6.63e-04 & 3.88  & 5.20e-04&4.02  \\
        4 & 8 & 1.40e-05 & 4.12  & 2.30e-07 & 5.03 & 4.14e-05 & 4.00 & 3.30e-05 & 3.94 \\
  \quad & 16 &8.80e-07 & 3.98  & 7.07e-09 & 5.01  & 2.59e-06 & 4.08   & 2.07e-06 & 4.05 \\
    \quad & 32 & 5.50e-08 & 4.00  & 2.22e-10 & 5.02  & 1.62e-07 & 4.05  & 1.30e-07 & 4.04 \\
  \hline
    \end{tabular}}}
\end{table}

{\it{Example 2.}} We consider the following equation with Dirichlet boundary condition:
\begin{equation}\label{model}
\begin{cases}
-(\alpha(x) u'(x))' + \beta(x) u'(x) +\gamma(x) u(x) = f(x), \quad x\in[0,1],\\
u(0) = u(1) = 0.
\end{cases}
\end{equation}
 In our experiments, we test the problems of variable coefficients and consider the following
  three  cases:
\begin{itemize}
\item Case 1: $\alpha(x)=e^x, \beta(x)=\cos(x), \gamma(x) = x;$
\item Case 2: $\alpha(x)=e^x, \beta(x)=0, \gamma(x) = x;$
\item Case 3: $\alpha(x)=e^x, \beta(x)=0, \gamma(x) = 0.$
\end{itemize}
The right-hand side function $f$ is chosen such that exact solution is
$$u(x) = \sin x(x^{12}-x^{11}).$$
We use the piecewise uniform meshes, which are constructed by equally dividing each interval $[0,\frac{2}{3}]$ and $[\frac{2}{3},1]$ into $N/2$ subintervals with $N=4, ... , 64$.

We present various approximation errors and the corresponding  convergence rates  in
Tables \ref{C1PG_exam2_k3}-\ref{C1PG_exam2_k4}
for the $C^1$ Petrov-Galerkin method, and  in Tables \ref{C1collo_exam2_k3}-\ref{C1collo_exam2_k4}
for the $C^1$ Gauss collocation method, for three different cases with $k=3,4$, respectively.
Again, we observe the same superconvergence results as those for the constant coefficients in Example 1, i.e., both errors $e_{un}$ and $e_{u'n}$ converge with a rate of $2k-2$, and
 the convergence  rates of  $e_{u}, e_{u'}, e_{u''}$ are $k+2, k+1, k,$ respectively.
 In other words, superconvergence results in Theorem \ref{theo:1} and Theorem \ref{theo:3}
 are still valid for the  case of variable coefficients.

\begin{table}[H]
   \caption{\small{ Errors and corresponding convergence rates for $C^1$ Petrov-Galerkin method, variable coefficients, $k=3$.}}\label{C1PG_exam2_k3}
   \vspace{0.1cm}
   \centerline{
   {\footnotesize
  \begin{tabular}{cccccccccc}
  \hline
  \quad& \quad& \multicolumn{2}{c}{$e_{un}$} & \multicolumn{2}{c}{$e_{u'n}$}
  & \multicolumn{2}{c}{$e_{u'}$} & \multicolumn{2}{c}{$e_{u''}$}\\
  \hline
  $k$ & $N$ & error & order & error & order & error & order & error & order \\
  \hline
  \multicolumn{10}{c}{Case 1}\\
   \hline
  \quad & 4 & 4.24e-04 & - & 4.42e-04 & - & 1.45e-02 & - & 4.98e-01& -\\
    \quad & 8 & 2.75e-05& 3.94 & 2.94e-05 & 3.91   & 1.38e-03 & 3.39 & 8.75e-02& 2.51 \\
  3 & 16 & 1.75e-06 & 3.97  & 1.87e-06 & 3.98  & 1.03e-04 &3.74  & 1.25e-02 & 2.81  \\
    \quad & 32 & 1.10e-07 & 4.00  & 1.17e-07 & 3.99  & 6.85e-06 & 3.91  & 1.63e-03 & 2.94  \\
    \quad & 64 & 6.86e-09	&4.00  & 7.34e-09 & 4.00   & 4.36e-07 & 3.97  & 2.05e-04 & 2.99 \\
  \hline
 \multicolumn{10}{c}{Case 2} \\
  \hline
  \quad & 4 & 3.39e-04 & - & 3.42e-04 & - & 1.37e-02 & - & 4.88e-01& -\\
    \quad & 8 & 2.25e-05 & 3.92  & 2.69e-05 & 3.67    &1.31e-03& 3.39  & 8.56e-02 & 2.51   \\
  3 & 16 &1.44e-06 & 3.96  & 1.88e-06& 3.84  & 9.74e-05 & 3.74  &1.22e-02 & 2.81   \\
    \quad & 32 & 9.03e-08&4.00  &1.23e-07 & 3.93  & 6.46e-06 & 3.91  & 1.58e-03& 2.95  \\
    \quad & 64 & 5.64e-09	&4.00  & 7.87e-09 & 3.97  &4.11e-07 & 3.98  & 2.00e-04 & 2.99    \\
  \hline
 \multicolumn{10}{c}{Case 3}\\
  \hline
  \quad & 4 &3.36e-04& - & 3.53e-04& -  & 1.37e-02 & - & 4.88e-01& -\\
    \quad & 8 & 2.26e-05 & 3.89  & 2.82e-05 & 3.65   & 1.30e-03 & 3.39 &	8.56e-02 & 2.51  \\
  3 & 16 & 1.44e-06& 3.97  & 1.97e-06& 3.84  & 9.72e-05 & 3.75  & 1.22e-02 &2.81  \\
    \quad & 32 & 9.05e-08	& 4.00  & 1.29e-07 & 3.93   & 6.45e-06 & 3.91  & 1.59e-03 & 2.95  \\
    \quad & 64 & 5.66e-09	& 4.00  & 8.26e-09& 3.97  & 4.09e-07& 3.98  & 2.00e-04 & 2.99  \\
  \hline
    \end{tabular}}}
\end{table}

\begin{table}[H]
   \caption{\small{ Errors and corresponding convergence rates for
   $C^1$ Petrov-Galerkin method, variable coefficients, $k=4$.}}\label{C1PG_exam2_k4}
   \vspace{0.1cm}
   \centerline{
   {\footnotesize
  \begin{tabular}{cccccccccccc}
  \hline
  \quad& \quad& \multicolumn{2}{c}{$e_{un}$} & \multicolumn{2}{c}{$e_{u'n}$}
  & \multicolumn{2}{c}{$e_u$} & \multicolumn{2}{c}{$e_{u'}$} & \multicolumn{2}{c}{$e_{u''}$}\\
  \hline
  $k$ & $N$ & error & order & error & order & error & order & error & order & error & order \\
  \hline
  \multicolumn{12}{c}{Case 1}\\
   \hline
  \quad & 4 & 2.56e-06 & - & 1.71e-06& -& 4.10e-05& - & 1.19e-03 & - & 6.21e-02& -\\
   \quad & 8 & 4.06e-08 & 5.98  & 3.46e-08 & 5.62  & 8.78e-07 & 5.55 & 4.83e-05 & 4.63  & 4.91e-03 & 3.66   \\
  4 & 16 & 6.25e-10 & 6.02  & 6.31e-10 & 5.78  & 1.35e-08 & 6.03  & 1.38e-06 & 5.13  & 2.90e-04 & 4.08  \\
    \quad & 32 & 1.05e-11& 5.89  & 1.03e-11 & 5.94  & 2.25e-10	&5.90  & 4.78e-08 & 4.85  & 1.94e-05 & 3.90 \\
    \quad & 64 & 1.63e-13 & 6.01  & 1.65e-13 & 5.96  & 4.02e-12 & 5.81  &1.52e-09& 4.98  & 1.24e-06 & 3.97 \\
  \hline
 \multicolumn{12}{c}{Case 2} \\
  \hline
  \quad & 4 & 1.25e-06& - & 8.07e-07 & -& 3.91e-05 & - &1.16e-03 & - & 6.07e-02 & -\\
   \quad & 8 & 2.01e-08 & 5.96  &2.10e-08& 5.27  & 8.31e-07	& 5.56& 	4.62e-05& 4.65  & 4.79e-03& 3.66  \\
  4 & 16 & 3.10e-10& 6.02  & 3.99e-10& 5.71  & 1.26e-08& 6.05 &1.33e-06	&5.12  & 2.84e-04 & 4.08 \\
    \quad & 32 &5.12e-12	&5.92  &6.70e-12 & 5.90 & 2.13e-10& 5.89 	&4.59e-08  & 4.86 &	1.89e-05	& 3.90  \\
    \quad & 64 &8.02e-14 & 6.00 & 1.05e-13 & 6.00  & 3.89e-12 & 5.77 	 & 1.46e-09&	4.98 & 1.21e-06 & 3.97  \\
  \hline
 \multicolumn{12}{c}{Case 3}\\
  \hline
  \quad & 4 & 9.26e-07& - & 5.75e-07 & -&3.93e-05 & - & 1.16e-03& - & 6.07e-02 & -\\
   \quad & 8 & 1.48e-08 & 5.96  & 1.54e-08 & 5.22 	&8.33e-07& 5.56 	&4.62e-05	&4.65 &4.79e-03 & 3.66  \\
  4 & 16 & 2.29e-10 & 6.02  & 2.95e-10 & 5.71 & 1.26e-08	& 6.05 & 1.33e-06	&5.11  & 2.84e-04& 4.08 \\
    \quad & 32 & 3.81e-12 & 5.91  & 4.94e-12 & 5.90  & 2.13e-10 & 5.89  &4.58e-08&4.86  & 1.89e-05 & 3.90  \\
    \quad & 64 & 5.74e-14	&6.05  & 9.17e-14 & 5.75 	& 3.91e-12&5.77 	&1.46e-09	&4.98  & 1.21e-06 & 3.97  \\
  \hline
    \end{tabular}}}
\end{table}

\begin{table}[H]
   \caption{\small{ Errors and corresponding convergence rates for $C^1$ Gauss collocation method, variable coefficients, $k=3$.}}\label{C1collo_exam2_k3}
   \vspace{0.1cm}
   \centerline{
   {\footnotesize
  \begin{tabular}{cccccccccc}
  \hline
    \quad& \quad& \multicolumn{2}{c}{$e_{un}$} & \multicolumn{2}{c}{$e_{u'n}$}
  & \multicolumn{2}{c}{$e_{u'}$} & \multicolumn{2}{c}{$e_{u''}$}\\
  \hline
  $k$ & $N$  & error & order & error & order & error & order & error & order \\
  \hline
  \multicolumn{10}{c}{Case 1}\\
   \hline
  \quad & 4 & 2.12e-03 & - & 3.30e-03& - & 1.29e-02 & - & 7.15e-02& -\\
    \quad & 8 & 1.43e-04 & 3.89  & 5.10e-04& 2.69  & 1.33e-03	& 3.28  & 1.35e-02 & 2.41  \\
  3 & 16 & 8.87e-06 & 4.01  & 4.71e-05 & 3.44   & 1.01e-04 & 3.72 &	2.03e-03&	2.73   \\
    \quad & 32 & 5.49e-07	&4.02  & 3.42e-06 & 3.78 & 6.67e-06& 3.92  & 2.76e-04	& 2.88  \\
    \quad & 64 & 3.42e-08	 & 4.00  & 2.24e-07& 3.93    & 4.22e-07 & 3.98 	& 3.59e-05 & 2.94 \\
  \hline
\multicolumn{10}{c}{Case 2}\\
  \hline
  \quad & 4 & 2.30e-03 & - & 2.97e-03 & - &1.42e-02 & - & 9.18e-02& -\\
    \quad & 8 & 1.56e-04&	3.89 &   5.01e-04 & 2.57   &1.45e-03& 3.30  & 1.71e-02&2.43    \\
  3 & 16 & 9.62e-06 & 4.02  & 4.72e-05& 3.41  &1.09e-04	& 3.72  & 2.55e-03& 2.74   \\
    \quad & 32 & 5.95e-07&4.01  & 3.45e-06 & 3.77   & 7.25e-06 & 3.92  & 3.46e-04 & 2.88   \\
    \quad & 64 & 3.71e-08& 4.00  & 2.26e-07 & 3.93  &4.59e-07 & 3.98  & 4.49e-05 & 2.94   \\
  \hline
\multicolumn{10}{c}{Case 3}\\
  \hline
  \quad & 4 & 2.37e-03& - &2.84e-03 & - &1.45e-02& - &9.13e-02& -\\
    \quad & 8 & 1.59e-04	& 3.89  & 4.85e-04& 2.55 & 1.47e-03& 3.30  &1.70e-02 & 2.43 \\
  3 & 16 & 9.82e-06& 4.02  & 4.60e-05 & 3.40   & 1.11e-04& 3.73  & 2.54e-03& 2.74   \\
    \quad & 32 & 6.08e-07&4.01  & 3.37e-06& 3.77  & 7.33e-06 & 3.92  & 3.45e-04 & 2.88    \\
    \quad & 64 & 3.79e-08&4.00  & 2.21e-07& 3.93   &4.64e-07 & 3.98  & 4.49e-05 & 2.94  \\
  \hline
    \end{tabular}}}
\end{table}

\begin{table}[H]
   \caption{\small{ Errors and corresponding convergence rates for $C^1$ Gauss collocation method, variable coefficients, $k=4$.}}\label{C1collo_exam2_k4}
   \vspace{0.1cm}
   \centerline{
   {\footnotesize
  \begin{tabular}{cccccccccccc}
  \hline
    \quad& \quad& \multicolumn{2}{c}{$e_{un}$} & \multicolumn{2}{c}{$e_{u'n}$}
  & \multicolumn{2}{c}{$e_u$} & \multicolumn{2}{c}{$e_{u'}$} & \multicolumn{2}{c}{$e_{u''}$}\\
  \hline
  $k$ & $N$ & error & order & error & order & error & order & error & order & error & order \\
  \hline
  \multicolumn{12}{c}{Case 1}\\
  \hline
  \quad & 4 & 1.45e-05 & - & 1.16e-04& -& 8.66e-05 & - & 1.00e-03 & - & 8.32e-03 & -\\
   \quad & 8 & 4.69e-07 & 4.95  & 3.01e-06 & 5.27 &1.53e-06	& 5.82 	&3.87e-05&4.69 & 7.68e-04 & 3.44  \\
  4 & 16 & 1.25e-08 & 5.23  & 4.84e-08& 5.96 & 1.64e-08 & 6.55  &1.14e-06 & 5.09& 5.70e-05&3.75   \\
    \quad & 32 &2.23e-10	&5.81  & 7.53e-10& 6.01  & 4.01e-10& 5.35  & 3.76e-08 & 4.92  & 3.73e-06& 3.94  \\
    \quad & 64 &3.61e-12	&5.95  & 1.18e-11 & 6.00  & 7.73e-12& 5.70  & 1.18e-09& 4.99  & 2.34e-07 & 4.00   \\
  \hline
\multicolumn{12}{c}{Case 2}\\
  \hline
  \quad & 4 & 1.60e-05 & - & 1.15e-04& - & 9.17e-05& - & 1.09e-03 & - & 1.06e-02 & -\\
   \quad & 8 & 4.82e-07& 5.05  &3.09e-06& 5.22  & 1.63e-06	&5.81  & 4.18e-05 & 4.70  & 9.85e-04 & 3.42 \\
  4 & 16 &1.31e-08& 5.20  & 5.06e-08 & 5.93  & 1.78e-08 & 6.52  & 1.25e-06& 5.06 	&7.21e-05&3.77   \\
    \quad & 32 & 2.35e-10	& 5.80  & 7.92e-10 & 6.00  & 4.08e-10 & 5.45  & 4.10e-08 & 4.93  & 4.68e-06& 3.94   \\
    \quad & 64 &3.80e-12& 5.95  & 1.24e-11 & 6.00&	7.92e-12 & 5.69  & 1.28e-09&5.00  & 2.93e-07& 4.00 \\
  \hline
\multicolumn{12}{c}{Case 3}\\
  \hline
  \quad & 4 & 1.61e-05 & - &1.15e-04 & -& 9.18e-05 & - & 1.09e-03 & - & 1.05e-02 & -\\
   \quad & 8 &4.84e-07& 5.06 &3.08e-06& 5.22  & 1.63e-06& 5.82 	&4.18e-05&4.70  & 9.86e-04& 3.42 \\
  4 & 16 & 1.32e-08 & 5.20  & 5.07e-08& 5.92 & 1.78e-08	& 6.52  & 1.25e-06 & 5.06 & 7.21e-05 & 3.77   \\
    \quad & 32 & 2.37e-10 & 5.80  & 7.96e-10 & 5.99  & 4.06e-10& 5.45  & 4.10e-08 & 4.93  & 4.68e-06& 3.95  \\
    \quad & 64 &3.84e-12	& 5.95  & 1.25e-11& 6.00 & 7.89e-12	& 5.69 	&1.28e-09	& 5.00  & 2.93e-07 & 4.00   \\
  \hline
    \end{tabular}}}
\end{table}

We also test the error $\|u_I-u_h\|_2$ in the above three cases of variable coefficients.
We list in Table \ref{C1uhuIh2_exam2_k3} and Table \ref{C1uhuIh2_exam2_k4} the numerical data  $\|u_I-u_h\|_2$ and
the convergence rate for $C^1$ Petrov-Galerkin approximation and Gauss collocation approximation with $k=3,4$.
 We observe that the convergence rate is always $k$ in different choices of variable coefficients, including the case $\beta=0$, Recall that for the constant coefficients case, the convergent rate of the $C^1$ Petrov-Galerkin approximation is $k+1$ when $\beta = 0$.

%compare the superconvergence property between the numerical solution $u_h$ and the Jacobi truncation projection $u_I$ of the exact solution under $H^2$ norm in different cases for different methods.
%Unlike the numerical results in {\it Example 1}, different choices of coefficients or numerical schemes would not change the corresponding convergence order.
%The error $\|u_h-u_I\|_2$ always achieves a convergence rate of $k$, as present in Table \ref{C1uhuIh2_exam2_k3} and Table \ref{C1uhuIh2_exam2_k4}.

\begin{table}[H]
   \caption{\small{$\|u_h-u_I\|_2$ and corresponding convergence rates, variable coefficients, $k=3$.}}\label{C1uhuIh2_exam2_k3}
   \vspace{0.1cm}
   \centerline{
   {\footnotesize
  \begin{tabular}{ccccccccc}
  \hline
\multirow{3}{*}{\quad}& \multicolumn{8}{c}{$\|u_h-u_I\|_2$} \\
\cline{2-9}
  &  $k$ & $N$  & error & order & error & order & error & order\\
  \cline{2-9}
   &\quad& \quad& \multicolumn{2}{c}{Case 1}&\multicolumn{2}{c}{Case 2} &\multicolumn{2}{c}{Case 3} \\
  \hline
  \quad&\quad & 4 & 2.35e-02& - & 1.89e-02 & - & 1.89e-02 &-\\
  \quad&\quad & 8 & 3.46e-03	& 2.77  & 2.82e-03	& 2.75  & 2.82e-03&	2.75   \\
   $C^1$ Petrov-Galerkin &3 & 16 &  4.49e-04&	2.94   & 3.67e-04& 2.94  & 3.67e-04&	2.94     \\
  \quad&\quad & 32 & 5.66e-05 & 2.99   &4.64e-05	& 2.99  &4.64e-05	& 2.99   \\
   \quad& \quad & 64 & 7.10e-06 & 3.00  & 5.81e-06& 3.00 &  5.81e-06&	3.00  \\
  \hline
 \quad& \quad & 4& 1.86e-01& - &1.93e-01& - &1.93e-01 &-\\
  \quad&\quad & 8& 2.65e-02	&2.81  & 2.75e-02 & 2.81  & 2.75e-02 & 2.81 \\
 $C^1$ Gauss collocation& 3 & 16&  3.37e-03 & 2.97   & 3.51e-03& 2.97    & 3.51e-03 & 2.97 \\
  \quad&\quad & 32 &4.22e-04 & 3.00    & 4.39e-04 & 3.00    & 4.39e-04	&3.00 \\
  \quad&\quad & 64 & 5.27e-05& 3.00  &5.49e-05	&3.00   & 5.49e-05	& 3.00 \\
\hline
    \end{tabular}}}
\end{table}

\begin{table}[H]
   \caption{\small{ $\|u_h-u_I\|_2$ and corresponding convergence rates,  variable coefficients, $k=4$.}}\label{C1uhuIh2_exam2_k4}
   \vspace{0.1cm}
   \centerline{
   {\footnotesize
  \begin{tabular}{ccccccccc}
  \hline
\multirow{3}{*}{\quad}& \multicolumn{8}{c}{$\|u_h-u_I\|_2$} \\
  \cline{2-9}
  &$k$ & $N$ & error & order & error & order & error & order\\
  \cline{2-9}
  &\quad& \quad& \multicolumn{2}{c}{Case 1}&\multicolumn{2}{c}{Case 2} &\multicolumn{2}{c}{Case 3}\\
  \hline
\quad&   \quad & 4 & 5.03e-03 & - & 4.48e-03& - & 4.48e-03&-\\
\quad&  \quad & 8 & 3.53e-04&3.83  & 3.17e-04	&3.82   & 3.17e-04&	3.82    \\
   $C^1$ Petrov-Galerkin &  4 & 16 & 2.24e-05 & 3.98  & 2.01e-05&	3.98   & 2.01e-05	& 3.98   \\
\quad&  \quad & 32 & 1.40e-06 & 4.00   & 1.26e-06 &  4.00  & 1.26e-06&	4.00   \\
  \quad&  \quad & 64 & 8.75e-08& 4.00   & 7.86e-08 & 4.00 & 7.86e-08& 	4.00   \\
  \hline
  \quad&\quad & 4 & 2.09e-02 & - & 2.18e-02& - & 2.18e-02&-\\
\quad &\quad & 8 & 1.32e-03& 3.99 & 1.38e-03& 3.99 & 1.38e-03 & 3.99 \\
 $C^1$ Gauss collocation& 4 & 16& 7.74e-05 & 4.09   &8.12e-05 & 4.08  &8.12e-05 & 4.08 \\
 \quad& \quad & 32 & 4.86e-06& 3.99  & 5.09e-06&4.00   & 5.09e-06& 4.00 \\
  \quad&\quad & 64& 3.06e-07	&3.99   & 3.20e-07& 3.99  & 3.20e-07 & 3.99  \\
\hline
    \end{tabular}}}
\end{table}

To sum up, superconvergence phenomena for problems with  variable coefficients
under non-uniform meshes still exist, and the superconvergence behavior for variable coefficients problems is similar with that for the
constant coefficients problems.

%
% \subsection{Superconvergence for  derivative value approximation}
%
%
%\begin{eqnarray*}
%  e'(x_i)&\lesssim &\int_{0}^{x_i} \alpha (u-u_h)''(x) dx=\sum_{j=1}^i \sum_{n=1}^{k-1}(u-u_h)''(g_{im})\omega_m+ h^{2k-1} \|u\|_{2k,\infty}\\
%       &\lesssim  &\sum_{j=1}^i  \sum_{n=1}^{k-1} \omega_m\left(\frac{\beta}{\alpha} (u-u_h)'(g_{im}) +\frac{\gamma }{\alpha} (u-u_h)(g_{im})  \right) + h^{2k-1} \|u\|_{2k,\infty}\\
%       &\lesssim &
%\end{eqnarray*}

\section{Conclusion}

 In this work, we have presented a unified approach to study superconvergence properties of $C^1$ Petrov-Galerkin
  and Gauss collocation methods for one-dimensional elliptic equations. Our main theoretical results include the proof of the $2k-2$ superconvergence rate for both solution and its first order derivative approximations at grid points, the $k+2$-th order  function value approximation at roots of the Jacobi polynomial $\hat J_{k+1}^{-2,-2}(x)$, the $k+1$-th order  derivative approximation at roots of $ (\hat J_{k+1}^{-2,-2})'(x)$ (i.e., the Lobatto points), and the $k$-th order second order derivative approximation at the Gauss points.
 An unexpected discovery is that the superconvergence rate of the first order derivative error at mesh points can reach as high as $2k-2$, which almost doubles the optimal convergence rate $k$. The superconvergence points for the second order derivative approximation is also novel.

  Our analysis indicate that for constant coefficients, the Gauss collocation method is essentially equivalent to the Petrov-Galerkin method up to practically neglect able numerical quadrature errors, see (\ref{bilinear:1}).  Indeed, we always use numerical quadrature instead of exact integration in practice.

Comparing with the traditional $C^0$ Galerkin method, the major gain of the $C^1$ Petrov-Galerkin method discussed in this work is the $2k-2$ convergence rate of the derivative approximation at nodes, with the sacrifice of function value convergence rate at nodes dropping from $2k$ to $2k-2$.

Comparing with the $C^1$ Galerkin method studied in \cite{Wahlbin}, the $C^1$ Petrov-Galerkin method discussed in this work has equal or better convergence rates in all respect. It seems that that the $L^2$ test function is superior to the $C^1$ test function. Therefore, the $C^1$ Petrov-Galerkin method is a method to recommend if one is also interested in derivative approximations.

 Based on the analysis, extension of our results to the higher dimensional tensor-product space is feasible.

\end{document}